\documentclass[12pt,article]{amsart}
\usepackage{mathrsfs}
\usepackage{amssymb}
\usepackage{amsfonts}
\usepackage{amsbsy}
\usepackage{latexsym}
\usepackage{amssymb,latexsym,amsmath,amsthm}
\usepackage{graphicx}
\usepackage{epsfig,psfrag}
\theoremstyle{plain}

 \theoremstyle{remark} 

\newtheorem {theo} {\bf Theorem} [section]

\newtheorem {lem} [theo] {\bf Lemma}

\newtheorem{rem}{\bf Remark}[section]

\newtheorem{comparison} {\bf Comparison} [section]

\numberwithin{equation}{section}
\begin{document}
\title[Framelets and nonuniform sampling approximations in Sobolev spaces]{Framelet  perturbation and  application to nouniform sampling  approximation for Sobolev space}
\author{Youfa Li}
\address{College of Mathematics and Information Science\\
Guangxi University,  Naning, China }
\email{youfalee@hotmail.com}
\author{Deguang Han}
\address{Department of Mathematics\\
University of Central Florida\\ Orlando, FL 32816}
\email{deguang.han@ucf.edu}
\thanks{Youfa Li is partially supported by Natural Science Foundation of China (Nos: 61561006, 11501132) and Natural Science Foundation of Guangxi (No: 2016GXNSFAA380049). Deguang Han  is partially supported by the NSF grant DMS-1403400.}

\date{\today}

\begin{abstract} The Sobolev space $H^{s}(\mathbb{R}^{d})$, where $s > d/2$, is an
important function space that  has many applications in various
areas of research. Attributed to the inertia of a measuring
instrument, it is desirable in sampling theory to reconstruct a
function by its nonuniform samples. In the present paper, we
investigate the problem of constructing the
 approximation to all the functions in $H^{s}(\mathbb{R}^{d})$ with nonuniform samples by utilizing dual
framelet systems  for  the Sobolev space pair
$(H^{s}(\mathbb{R}^{d}), H^{-s}(\mathbb{R}^{d}))$. We first
establish the convergence rates of the framelet series in
$(H^{s}(\mathbb{R}^{d}), H^{-s}(\mathbb{R}^{d}))$, and  then
construct the framelet approximation operator holding for the entire
space $H^{s}(\mathbb{R}^{d})$. Using the  approximation operator,
any function in $H^{s}(\mathbb{R}^{d})$ can be approximated at the
exponential rate with respect to the scale level. We examine the
stability property for the perturbations of the framelet
approximation operator with respect to shift parameters, and obtain
an estimate bound for the perturbation error. Our result  shows that
under the condition $s > d/2$, the approximation operator is robust
to the shift perturbation. These results are used to establish the
nonuniform sampling approximation for every function in
$H^{s}(\mathbb{R}^{d})$. In particular, the new nonuniform sampling
approximation error is robust to the jittering of the samples.
\end{abstract}
\maketitle

\section{Introduction}\label{section1}
Sampling is a fundamental tool for the conversion between an
analogue signal and its digital form (A/D). The most classical
sampling theory is the Whittaker-Kotelnikov-Shannon (WKS) sampling
theorem \cite{C. Shannon1,C. Shannon2}, which states   that a
bandlimited signal can be perfectly reconstructed  if it is sampled
at a rate greater than its Nyquist frequency.
The WKS sampling theorem holds only for bandlimited signals. In
order to extend the sampling theorem to non-bandlimited signals,
researchers have established various sampling theorems for many
other function spaces. Such examples include  the sampling theory
for shift-invariant subspaces (c.f.
\cite{Aldroubi1,Aldroubi2,Qiyu1,df1,df0}), for reproducing kernel
subspaces of $L^{2}(\mathbb{R}^{d})$ (c.f.
\cite{HD2,Qiyu1,Qiyu0,Wenchang2,HD1}), and for subspaces from the
generalized sinc function (c.f. \cite{CQL}).

For any $s\in\mathbb{R}$,  the Sobolev space $H^{s}(\mathbb{R}^{d})$
is defined as
\begin{align}\label{defi_s} H^{s}(\mathbb{R}^{d})=\Big\{f: \int_{\mathbb{R}^{d}}|\widehat{f}(\xi)|^{2} (1+||\xi||_{2}^{2})^{s}d\xi<\infty\Big\},\end{align}
 where
$\widehat{f}(\xi):=\int_{\mathbb{R}^{d}}f(x)e^{-ix\cdot \xi}dx$ is
the Fourier transform of $f$. When $s>d/2$, the function theory of
$H^{s}(\mathbb{R}^{d})$  has been extensively applied to various
problems such as the boundedness of the Fourier multiplier operator
\cite{youyong3, youyong2, youyong1}, viscous shallow water system
\cite{youyong4,youyong5}, PDE \cite{youyong6}, and signal analysis
\cite{saPP, Matla}. On the other hand, it will be seen in Theorem
\ref{Theorem x2} or Remark \ref{lubax} that the condition $s>d/2$ is
necessary to guarantee that the approximation system in
$H^{s}(\mathbb{R}^{d})$ is robust to the perturbation of the shift
parameters, which is crucial for our construction of nonuniform
sampling approximation. Moreover, it is easy to check that many
frequently used  spaces such as the bandlimited function space,
wavelet subspaces \cite{Chui,AFJ04}  and the cardinal B-spline
subspaces \cite{Chui,Hanbin1} (in which the generator is continuous)
are all contained in $H^{s}(\mathbb{R}^{d})$. Readers are referred
to Han and Shen \cite{Hanbin1} for the Sobolev smoothness  of box
splines.

 Since
$H^{s}(\mathbb{R}^{d})$ and $H^{-s}(\mathbb{R}^{d})$ are isometric
under a mapping provided in the proof of \cite[Proposition
2.1]{Hanbin1}, we can treat $H^{-s}(\mathbb{R}^{d})$ as the dual
space of $H^{s}(\mathbb{R}^{d})$. It can  be seen in Theorem
\ref{Theorem x4} or \cite{Youfa1} that by using   special dual
framelets in $(H^{s}(\mathbb{R}^{d}), H^{-s}(\mathbb{R}^{d}))$, the
inner products can expressed directly by the values of functions,
which makes the sampling approximation possible. In the
one-dimensional $(d =1$) case,  a uniform sampling theorem  for all
the functions in $H^{s}(\mathbb{R})$,  where $s>1/2$, was
established by Li and Yang  in \cite{Youfa1}. Attributed to the
inertia of a measuring instrument, the samples we acquire may well
be jittered and thus nonuniform \cite{Wenchangui,Wenchangui1,Qiyu0}.
Therefore it seems necessary to establish a theory for nonuniform
sampling  for  all the functions in $H^{s}(\mathbb{R^{d}}^{d})$. The
purpose of this paper is to build such a theory by using a pair of
dual framelet system for the Sobolev space pair
$(H^{s}(\mathbb{R}^{d}), H^{-s}(\mathbb{R}^{d}))$ $(s > d/2$).

We first  introduce some necessary notations and terminologies  for
framelets in Sobolev spaces. More details can be found in Han and
Shen \cite{Hanbin1} where the dual framelets for the dual pair
$(H^{s}(\mathbb{R}^{d}), H^{-s}(\mathbb{R}^{d}))$ were first
introduced. We remark that, comparing  with those in
$L^{2}(\mathbb{R}^{d})$, the framelet in $H^{s}(\mathbb{R}^{d})$
does not necessarily have vanishing moment. Therefore the
construction of the framelet system seems much more easier in this
case. Readers are referred to \cite{Hanbin2,Hanbin3} for Han's
continuing work in the distribution spaces.

By \eqref{defi_s}, $H^{s}(\mathbb{R}^{d})$ is equipped with the
inner product $\langle\cdot,\cdot\rangle_{H^{s}(\mathbb{R}^{d})}$
defined by
\begin{align}\label{neiji} \langle f,g\rangle_{H^{s}(\mathbb{R}^{d})}=\frac{1}{(2\pi)^{d}}\int_{\mathbb{R}^{d}}\widehat{f}(\xi)\overline{\widehat{g}(\xi)}(1+||\xi||_{2}^{2})^{s}d\xi,
\ \ \ \ \forall f,g\in H^{s}(\mathbb{R}^{d}),\end{align} where
$\overline{\widehat{g}}$ is the complex conjugate.
 The deduced norm
$||\cdot||_{H^{s}(\mathbb{R}^{d})}$ of
$\langle\cdot,\cdot\rangle_{H^{s}(\mathbb{R}^{d})}$ is naturally
given by
$$||f||_{H^{s}(\mathbb{R}^{d})}=\frac{1}{(2\pi)^{d/2}}\Big(\int_{\mathbb{R}^{d}}|\widehat{f}(\xi)|^{2}(1+||\xi||_{2}^{2})^{s}d\xi\Big)^{1/2}.$$
It is easy to check that  the bilinear functional $\langle \cdot,
\cdot \rangle: (H^{s}(\mathbb{R}^{d}),
H^{-s}(\mathbb{R}^{d}))\longrightarrow \mathbb{C}$ defined by
$$\langle f, g
\rangle=\frac{1}{(2\pi)^{d}}\int_{\mathbb{R}^{d}}\widehat{f}(\xi)\overline{\widehat{g}(\xi)}d\xi,
\ \forall f\in H^{s}(\mathbb{R}^{d}),  g\in H^{-s}(\mathbb{R}^{d})$$
satisfies $|\langle f, g \rangle|\leq
||f||_{H^{s}(\mathbb{R}^{d})}||g||_{H^{-s}(\mathbb{R}^{d})}.$
 Straightforward
observation on \eqref{defi_s} gives  that
$H^{s_{1}}(\mathbb{R}^{d})$ $\supseteq H^{s_{2}}(\mathbb{R}^{d})$ if
and only if $s_{1}\leq s_{2}$. When $s=0$, we have that
$H^{0}(\mathbb{R}^{d})=L^{2}(\mathbb{R}^{d})$, and the corresponding
norm $||\cdot||_{H^{0}(\mathbb{R}^{d})}$ is the usuall $L_{2}$-norm
$||\cdot||_{2}$. In what follows we will use the same norm
denotation $||\cdot||_{2}$ for  $L^{2}(\mathbb{R}^{d})$ and  the
Euclidean space $\mathbb{R}^{d}.$ The two norms can be easily
identified from the
 context. For any $f\in H^{s}(\mathbb{R}^{d})$, define its bracket product $[f,
f]_{s}$ as
\begin{align}\label{opq}[f,
f]_{s}(\xi):=\sum_{k\in\mathbb{Z}^{d}}|\widehat{f}(\xi+2k\pi)|^{2}(1+||\xi+2k\pi||_{2}^{2})^{s}.\end{align}
When $f$ is compactly supported,  we have that  $[f, f]_{s}\in
L_{\infty}(\mathbb{R}^{d})$.  We refer to Han's method \cite{WEZ}
for more information about the bracket product estimation.

A $d\times d$ integer matrix $M$ is referred to as  a dilation
matrix if all its eigenvalues are strictly larger than $1$ in
modulus. Throughout this paper, we are interested in the case that
$M$ is isotropic. Specifically, $M$ is similar to
$\mbox{diag}(\lambda_{1},\lambda_{2},\cdots,\lambda_{d})$ with
$|\lambda_{k}|=m:=|\det M|^{1/d}$ for $k=1, 2, \ldots, d.$ Denote by
$\Gamma_{M^{T}}$ the complete set of representatives of distinctive
cosets of the quotient group
$[(M^{T})^{-1}\mathbb{Z}^{d}]/\mathbb{Z}^{d}.$
%
Suppose that $\phi\in H^{s}(\mathbb{R}^{d}), s\in\mathbb{R},$ is an
$M$-refinable function given by
\begin{align}\label{yy1}\widehat{\phi}(M^{T}\cdot)=\widehat{a}(\cdot)\widehat{\phi}(\cdot),\end{align}
where $\widehat{a}(\cdot):=\sum_{k\in\mathbb{Z}^{d}}a[k]e^{ik\cdot}$
is referred to as the mask symbol of $\phi$, and
 $\{\psi^{\ell}\}^{L}_{\ell=1}$ is a set of  wavelet functions
 defined by
\begin{align}\label{yy3}\widehat{\psi^{\ell}}(M^{T}\cdot)=\widehat{b^{\ell}}(\cdot)\widehat{\phi}(\cdot),\end{align}
where the $2\pi \mathbb{Z}^{d}$-periodic trigonometric polynomial
$\widehat{b^{\ell}}(\cdot)$ is the mask symbol of $\psi^{\ell}$. Now
a wavelet system $X^{s}(\phi;$ $\psi^{1},\ldots,\psi^{L})$ in
$H^{s}(\mathbb{R}^{d})$ is defined as
\begin{align}\begin{array}{lllll}\label{lisan}
X^{s}(\phi;\psi^{1},\ldots,\psi^{L})&:=\{\phi_{0,k}:k\in
\mathbb{Z}^{d}\}
\\&\quad \cup\{\psi^{\ell,s}_{j,k}:k\in\mathbb{Z}^{d}, j\in
\mathbb{N}_{0},\ell=1,\ldots,L\},
\end{array}\end{align}
where $\phi_{0,k}=\phi(\cdot-k)$,
$\psi^{\ell,s}_{j,k}=m^{j(d/2-s)}\psi^{\ell}(M^{j}\cdot-k)$ and
$\mathbb{N}_{0}:=\mathbb{N}\cup\{0\}$. If there exist  two positive
constants $C_{1}$ and $C_{2}$ such that
\begin{align}\label{hanbin3}\begin{array}{lllll}
C_{1}||f||_{H^{s}(\mathbb{R}^{d})}^{2}\displaystyle\leq
\sum_{k\in\mathbb{Z}^{d}}|\langle
f,\phi_{0,k}\rangle_{H^{s}(\mathbb{R}^{d})}|^{2}
 +\sum^{L}_{\ell=1}\sum_{j\in\mathbb{N}_{0}}\sum_{k\in
\mathbb{Z}^{d}} |\langle f,\psi^{\ell,s}_{j,k}
\rangle_{H^{s}(\mathbb{R}^{d})}|^{2} \leq
C_{2}||f||_{H^{s}(\mathbb{R}^{d})}^{2}
\end{array}\end{align}
holds for every $ f\in H^{s}(\mathbb{R}^{d}),$ then we say that
 $X^{s}(\phi;\psi^{1},\ldots,\psi^{L})$ is an $M$-framelet system in
$H^{s}(\mathbb{R}^{d})$. If there exists another $M$-framelet system
$X^{-s}(\widetilde{\phi};\widetilde{\psi}^{1},\ldots,\widetilde{\psi}^{L})$
in $H^{-s}(\mathbb{R}^{d})$ such that for any $ f\in
H^{s}(\mathbb{R}^{d})$ and $ g\in H^{-s}(\mathbb{R}^{d})$, there
holds
\begin{align}\label{hanbin4}\langle f,g\rangle=\sum_{k\in \mathbb{Z}^{d}}\langle
\phi_{0,k}, g\rangle\langle f, \widetilde{\phi}_{0,k}\rangle+
\sum^{L}_{\ell=1}\sum_{j\in \mathbb{\mathbb{N}}_{0}}\sum_{k\in
\mathbb{Z}^{d}} \langle \psi^{\ell,s}_{j,k}, g \rangle\langle f,
\widetilde{\psi}^{\ell,-s}_{j,k} \rangle,
\end{align}
then we say that $X^{s}(\phi;\psi^{1},\ldots,\psi^{L})$ and
$X^{-s}(\widetilde{\phi};\widetilde{\psi}^{1},\ldots,\widetilde{\psi}^{L})$
form a pair of dual $M$-framelet systems in
$(H^{s}(\mathbb{R}^{d}),H^{-s}(\mathbb{R}^{d}))$. For any function
$f\in H^{s}(\mathbb{R}^{d})$, it follows
 from \eqref{hanbin4} that
\begin{align}\begin{array}{lll}\label{Albert1}
f=\displaystyle\sum_{k\in \mathbb{Z}^{d}}\langle f,
\widetilde{\phi}_{0,k}\rangle \phi_{0,k}+
\sum^{L}_{\ell=1}\sum_{j\in \mathbb{\mathbb{N}}_{0}}\sum_{k\in
\mathbb{Z}^{d}} \langle f, \widetilde{\psi}^{\ell,-s}_{j,k} \rangle
\psi^{\ell,s}_{j,k}.
\end{array}\end{align}

Our goal is to construct the nonuniform sampling approximation to
any function $f\in H^{s}(\mathbb{R}^{d}),$ $ s>d/2$.  Our
 approximation will be derived from the
truncation form $\mathcal{S}_{\phi}^{N}f$ of the series in
\eqref{Albert1}, defined by
\begin{align}\label{rt} \mathcal{S}_{\phi}^{N}f:=\sum_{k\in
\mathbb{Z}^{d}}\langle f, \widetilde{\phi}_{0,k}\rangle
\phi_{0,k}+\sum^{L}_{\ell=1}\sum^{N-1}_{j=0}\sum_{k\in
\mathbb{Z}^{d}} \langle f, \widetilde{\psi}^{\ell,-s}_{j,k} \rangle
\psi^{\ell,s}_{j,k},\end{align} where $N$
 is sufficiently large.  The first
natural problem is how to
 estimate the approximation error $||(I-\mathcal{S}_{\phi}^{N})f||$, where $I$ is the identity operator, and $||\cdot||$ is the desired norm.
  When $f$ belongs to the Schwartz class of functions,
 the estimate of $||(I-\mathcal{S}_{\phi}^{N})f||_{2}$ was given in  \cite[Theorem 16]{Skopina}.
In  \cite{fagep},  the approximation error
$||(I-\mathcal{S}_{\phi}^{N})f||_{H^{s}(\mathbb{R}^{d})}$ was
 estimated when $f$ satisfies  $$|\widehat{f}(\xi)|\leq C(1+||\xi||_{2})^{\frac{-d-\alpha}{2}} \ \hbox{for} \ \hbox{every}\ \xi\in \mathbb{R}^{d},$$
with $\alpha>0$ and  a constant  $C$ being dependent on $f$. When
 the framelet system
$X^{-s}(\widetilde{\phi};\widetilde{\psi}^{1},\ldots,\widetilde{\psi}^{L})$
belongs to $L^{2}(\mathbb{R}^{d})$, for
 any $f\in H^{s}(\mathbb{R}^{d})$, the estimate of
 $||(I-\mathcal{S}_{\phi}^{N})f||_{2}$ was obtained  in  \cite{Hanbin4} and  \cite{Skopina}.  In the present paper,
by  using a special pair of dual framelet systems
 $X^{s}(\phi;\psi^{1},\ldots,\psi^{L})$ and
$X^{-s}(\widetilde{\phi};\widetilde{\psi}^{1},\ldots,\widetilde{\psi}^{L})$,
we aim at constructing the nonuniform sampling approximation to all
the functions in $H^{s}(\mathbb{R}^{d})$, where the system
$X^{-s}(\widetilde{\phi};\widetilde{\psi}^{1},\ldots,\widetilde{\psi}^{L})$
in $H^{-s}(\mathbb{R}^{d})$ does not actually  belong  to
$L^{2}(\mathbb{R}^{d})$. Therefore  in order to construct the
sampling approximation in this setting, we need to estimate
$||(I-\mathcal{S}_{\phi}^{N})f||$ for any $f\in
H^{s}(\mathbb{R}^{d})$,   not requiring
$X^{-s}(\widetilde{\phi};\widetilde{\psi}^{1},\ldots,\widetilde{\psi}^{L})$
being in  $L^{2}(\mathbb{R}^{d})$.  Such an estimate will be
presented in Theorem \ref{Theorem x1}.

It  will be seen in \eqref{integral1} and \eqref{suanzidigyi}
 that the nonuniformity of samples is
substantially derived from the perturbation of shifts of the
sampling system $\{\Delta^{-s}_{N,k}\}_{k\in
\mathbb{Z}^{d}}\subseteq H^{-s}(\mathbb{R}^{d})$, where $\Delta\in
H^{-s}(\mathbb{R}^{d})$ is a special refinable function to be
defined in \eqref{caiyanghanshu},  and $\Delta^{-s}_{N,k}$ will be
given via \eqref{dingyi}. Thus, in order to construct the nonuniform
sampling approximation,  we need to establish the estimate  for the
perturbation error of $\mathcal{S}_{\phi}^{N}f$
 when the shifts of
 $\{\widetilde{\phi}^{-s}_{N,k}\}_{k\in
\mathbb{Z}^{d}}$ are perturbed, where $\widetilde{\phi}$ is any
refinable function in $H^{-s}(\mathbb{R}^{d})$. Our second main
Theorem \ref{Theorem xe4} establishes such an error estimate.

In the Section 3 we present a main application of our two main
results. By using a pair of dual framelets for
$(H^{s}(\mathbb{R}^{d}), H^{-s}(\mathbb{R}^{d})),$   we are able to
construct the nonuniform sampling approximation to any function in
$H^{s}(\mathbb{R}^{d})$ where $s>d/2$.   We also compare  the main
results  of this paper with  the existing ones in the literature,
and present two simulation examples to demonstrate the approximation
efficiency in numerical experiments.


\section{Perturbed  framelet approximation system in Sobolev space}

In this section we will first  estimate the convergence rate of  the
coefficient sequence $\big\{\langle f,
\widetilde{\psi}^{\ell,-s}_{j,k} \rangle\big\}$ in \eqref{Albert1}.
Based on the convergence rate estimation, the approximation error
$(I-\mathcal{S}_{\phi}^{N})f$ for any $f\in H^{s}(\mathbb{R}^{d})$
will be estimated, not requiring
$X^{-s}(\widetilde{\phi};\widetilde{\psi}^{1},\ldots,\widetilde{\psi}^{L})$
being in  $L^{2}(\mathbb{R}^{d})$. Moreover, the corresponding
perturbation error of $\mathcal{S}_{\phi}^{N}f$ will be given when
the shifts of approximation system
$\{\widetilde{\phi}^{-s}_{N,k}\}_{k\in \mathbb{Z}^{d}}$ are
perturbed, where $\widetilde{\phi}^{-s}_{N,k}$ will be  defined in
\eqref{dingyi}.

\subsection{Framelet approximation system}
\label{section2}  For any $\alpha:=(\alpha_{1}, \alpha_{2}, \ldots,
\alpha_{d})\in \mathbb{N}^{d}_{0}$ and $x:=(x_{1}, x_{2}, \ldots,
x_{d})\in\mathbb{R}^{d}$, define
$x^{\alpha}=\prod^{d}_{k=1}x^{\alpha_{k}}_{k}$. For any function $f:
\mathbb{R}^{d}\longrightarrow \mathbb{C}$, its $\alpha$th partial
derivative $\frac{\partial^{\alpha}}{\partial x^{\alpha}}f$ is
defined as
$$\frac{\partial^{\alpha}}{\partial x^{\alpha}}f=\frac{\partial_{1}^{\alpha_{1}}\partial_{2}^{\alpha_{2}}\cdots
\partial_{d}^{\alpha_{d}}}{\partial x_{1}^{\alpha_{1}}\partial x_{2}^{\alpha_{2}}\cdots
\partial x_{d}^{\alpha_{d}}}f.$$ We say that a function  $f:
\mathbb{R}^{d}\longrightarrow \mathbb{C}$ has $\kappa+1
(\in\mathbb{N})$ vanishing moments  if
$$\frac{\partial^{\alpha}}{\partial x^{\alpha}}\widehat{f}(0)=0$$
for any  $\alpha\in\mathbb{N}^{d}_{0}$ such that $||\alpha||_{1}\leq
\kappa$, where $||\cdot||_{1}$ is the $1$-norm of a vector.

In the proof of Theorem \ref{Theorem x1}, we will see that the
crucial task for estimating
$||(I-\mathcal{S}_{\phi}^{N})f||_{H^{s}(\mathbb{R}^{d})}$ is  to
estimate the convergence rate of the coefficient sequence $\{\langle
f, \widetilde{\psi}^{\ell,-s}_{j,k} \rangle\}_{j,k,\ell}$ in
\eqref{Albert1}. As such, we first estimate the convergence rate of
$\{\langle f, \widetilde{\psi}^{\ell,-s}_{j,k} \rangle\}_{j,k,\ell}$
in Lemma \ref{Lemma 2.1}.

\begin{lem}\label{Lemma 2.1}
Let $s>0$ and $\widetilde{\phi}\in H^{-s}(\mathbb{R}^{d})$ be
$M$-refinable. Moreover, suppose that $\widetilde{\phi}\in
H^{-t}(\mathbb{R}^{d})$ where $0<t<s$.  A wavelet function
$\widetilde{\psi}$ given   by
$\widehat{\widetilde{\psi}}(M^{T}\cdot)=\widehat{\widetilde{b}}(\cdot)\widehat{\widetilde{\phi}}(\cdot)$
has $\kappa+1$ vanishing moments, where $\widehat{\widetilde{b}}$ is
a $2\pi\mathbb{Z}^{d}$-periodic trigonometric polynomial,
$\kappa\in\mathbb{N}_{0}$ and $\kappa+1>t$.
 Then there exists a positive constant $G(\widehat{\widetilde{b}}, s, t)$ such
 that for any $ f\in H^{\varsigma}(\mathbb{R}^{d})$, it holds
\begin{align} \label{lanlan}\displaystyle\sum^{\infty}_{j=N}\sum_{k\in \mathbb{Z}^{d}}|\langle
f, \widetilde{\psi}^{-s}_{j,k}\rangle|^{2}\leq
G(\widehat{\widetilde{b}}, s,
t)||f||^{2}_{H^{\varsigma}(\mathbb{R}^{d})} m^{-2N\eta_{\kappa+1}(s,
\varsigma)}, \end{align} where $t<s<\varsigma<\kappa+1$ and
\begin{align}\label{rt609} \eta_{\kappa+1}(s, \varsigma):=(\kappa+1-s)(\varsigma-s)/(\kappa+1+\varsigma-s).\end{align}
\end{lem}
\begin{proof}
By the vanishing moment property  of $\widetilde{\psi}$, there
exists a positive constant $C_{0}(\widehat{\widetilde{b}})$ such
that
$$|\widehat{\widetilde{b}}(\xi)|\leq C_{0}(\widehat{\widetilde{b}})||\xi||^{\kappa+1}_{2} \ \mbox{for} \ \mbox{any} \ \xi\in \mathbb{R}^{d}.$$
By the similar procedure as \cite[(2.9, 2.10)]{Li0}, we have
\begin{align}\begin{array}{lllll}\label{df3}
\displaystyle\sum^{\infty}_{j=N}\sum_{k\in \mathbb{Z}^{d}}|\langle f, \widetilde{\psi}^{-s}_{j,k}\rangle|^{2}\\
\leq\displaystyle
\frac{m^{d}||[\widehat{\widetilde{\phi}},\widehat{\widetilde{\phi}}]_{-t}||_{L^{\infty}(\mathbb{R}^{d})}}{(2\pi)^{d}}
\displaystyle\int_{\mathbb{R}^{d}}
|\widehat{f}(\xi)|^{2}\sum^{\infty}_{j=N}m^{2js}|\widehat{\widetilde{b}}((M^{T})^{-j-1}\xi)|^{2}(1+||(M^{T})^{-j-1}\xi||^{2}_{2})^{t}d\xi.
\end{array}\end{align}
The integral in \eqref{df3} is split into the two parts as follows,
\begin{align}\begin{array}{lllll}\notag
I_{N,1}=\displaystyle\int_{||\xi||_{2}<m^{N\nu}}
|\widehat{f}(\xi)|^{2}\sum^{\infty}_{j=N}m^{2js}|\widehat{\widetilde{b}}((M^{T})^{-j-1}\xi)|^{2}(1+||(M^{T})^{-j-1}\xi||^{2}_{2})^{t}
d\xi,
\end{array}\end{align}
and
\begin{align}\begin{array}{lllll}
I_{N,2}=\displaystyle\displaystyle \notag\int_{||\xi||_{2}\geq
m^{N\nu}}
|\widehat{f}(\xi)|^{2}\sum^{\infty}_{j=N}m^{2js}|\widehat{\widetilde{b}}((M^{T})^{-j-1}\xi)|^{2}(1+||(M^{T})^{-j-1}\xi||^{2}_{2})^{t}d\xi,
\end{array}\end{align}
where $\nu\in(0,1)$ will be optimally selected.  At first, the term
$I_{N,1}$ is estimated as follows,
\begin{align}\begin{array}{lllll}\label{df1}
\displaystyle I_{N,1}&\leq\displaystyle
2^{t}C_{0}(\widehat{\widetilde{b}})^{2}\int_{||\xi||_{2}<m^{N\nu}}
|\widehat{f}(\xi)|^{2}(1+||\xi||_{2}^{2})^{s}\sum^{\infty}_{j=N}||(M^{T})^{-j-1}\xi)||_{2}^{2\kappa+2}
m^{2js}  d\xi\\
&=\displaystyle
2^{t}C_{0}(\widehat{\widetilde{b}})^{2}\int_{||\xi||_{2}<m^{N\nu}}
|\widehat{f}(\xi)|^{2}(1+||\xi||_{2}^{2})^{s}\sum^{\infty}_{j=N}||m^{-j-1}\xi||_{2}^{2\kappa+2}
m^{2js} d\xi\\
&\leq\displaystyle
2^{t}C_{0}(\widehat{\widetilde{b}})^{2}\int_{||\xi||_{2}<m^{N\nu}}
|\widehat{f}(\xi)|^{2}(1+||\xi||_{2}^{2})^{s}m^{-2(\kappa+1)}m^{2N\nu(\kappa+1)}\sum^{\infty}_{j=N}m^{-2j(\kappa+1-s)} d\xi\\
&=\displaystyle2^{t}C_{0}(\widehat{\widetilde{b}})^{2}m^{-2(\kappa+1)}\frac{m^{-2N[(\kappa+1)(1-\nu)-s]}}{1-m^{-2(\kappa+1-s)}}
\int_{||\xi||_{2}<m^{N\nu}}
|\widehat{f}(\xi)|^{2}(1+||\xi||_{2}^{2})^{s}d\xi\\
&\leq\displaystyle2^{t}(2\pi)^{d}C_{0}(\widehat{\widetilde{b}})^{2}m^{-2(\kappa+1)}\frac{m^{-2N[(\kappa+1)(1-\nu)-s]}}{1-m^{-2(\kappa+1-s)}}
||f||^{2}_{H^{s}(\mathbb{R}^{d})}.
\end{array}\end{align}
We next estimate $I_{N,2}$.
 For any $\xi\in\mathbb{R}^{d}$,
it follows from \cite[Lemma 2.2]{Li0} that
\begin{align} \label{quantity} \sum^{\infty}_{j=N}m^{2js}|\widehat{\widetilde{b}}((M^{T})^{-j-1}\xi)|^{2}(1+||\xi||^{2}_{2})^{-s}(1+||(M^{T})^{-j-1}\xi||^{2}_{2})^{t}\leq
C_{1}(\widehat{\widetilde{b}}, s,t),\end{align} where
\begin{align}\label{zjg} C_{1}(\widehat{\widetilde{b}}, s,t):=\frac{||\widehat{b}(\xi)||_{L^{\infty}(\mathbb{R}^{d})}^{2}2^{t}}{m^{2(s-t)}-1}+
\frac{C_{0}(\widehat{\widetilde{b}})^{2}}{1-m^{-2(\kappa+1-s)}}.\end{align}
Then
\begin{align}\begin{array}{lllll}\label{df2}\displaystyle
I_{N,2}&=\displaystyle\int_{||\xi||_{2}\geq m^{N\nu}}
|\widehat{f}(\xi)|^{2}(1+||\xi||^{2}_{2})^{s}\Big[\sum^{\infty}_{j=N}m^{2js}|\widehat{\widetilde{b}}((M^{T})^{-j-1}\xi)|^{2}(1+||\xi||^{2}_{2})^{-s}\\
&\quad  \times (1+||(M^{T})^{-j-1}\xi||^{2}_{2})^{t}\Big]d\xi\\
&\leq\displaystyle C_{1}(\widehat{\widetilde{b}},
s,t)\int_{||\xi||_{2}\geq m^{N\nu}}
|\widehat{f}(\xi)|^{2}(1+||\xi||^{2}_{2})^{s}d\xi\\
&\leq\displaystyle C_{1}(\widehat{\widetilde{b}},
s,t)m^{-2N\nu(\varsigma-s)}\int_{||\xi||_{2}\geq m^{N\nu}}
|\widehat{f}(\xi)|^{2}(1+||\xi||^{2}_{2})^{\varsigma}d\xi\\
&\leq\displaystyle (2\pi)^{d}C_{1}(\widehat{\widetilde{b}},
s,t)m^{-2N\nu(\varsigma-s)}||f||^{2}_{H^{\varsigma}(\mathbb{R}^{d})}.
\end{array}\end{align}
 By \eqref{df3}, \eqref{df1} and
\eqref{df2}, we obtain
$$\sum^{\infty}_{j=N}\sum_{k\in \mathbb{Z}}|\langle f, \widetilde{\psi}^{-s}_{j,k}\rangle|^{2}
=\mbox{O}\Big(m^{-2N\min\big\{(\kappa+1)(1-\nu)-s, \ \nu
(\varsigma-s)\big\}}\Big).$$ It is easy to prove that  the
convergence rate of
$$\sum^{\infty}_{j=0}\sum_{k\in \mathbb{Z}}|\langle f,
\widetilde{\psi}^{-s}_{j,k}\rangle|^{2}$$ reaches the optimal
converging order $\mbox{O}(m^{-2N\eta_{\kappa+1}(s, \varsigma) })$
if selecting $\nu:=(\kappa+1-s)/(\kappa+1+\varsigma-s)$, where
$\eta_{\kappa+1}$ is defined in \eqref{rt609}. Define
\begin{align}\label{gbound} G(\widehat{\widetilde{b}}, s, t):=m^{d}||[\widehat{\widetilde{\phi}},\widehat{\widetilde{\phi}}]_{-t}||_{L^{\infty}(\mathbb{R}^{d})}\Big(\frac{2^{t}C_{0}(\widehat{\widetilde{b}})^{2}m^{-2(\kappa+1)}}{1-m^{-2(\kappa+1-s)}}+C_{1}(\widehat{\widetilde{b}}, s, t)\Big).\end{align}
It follows from \eqref{df3}, \eqref{df1} and \eqref{df2} that
\begin{align}\begin{array}{lllll}\label{ucf}
\displaystyle\sum^{\infty}_{j=N}\sum_{k\in \mathbb{Z}^{d}}|\langle
f, \widetilde{\psi}^{-s}_{j,k}\rangle|^{2} \leq
G(\widehat{\widetilde{b}}, s, t) m^{-2N\eta_{\kappa+1}(s,
\varsigma)}||f||^{2}_{H^{\varsigma}(\mathbb{R}^{d})}.
\end{array}\end{align}
The proof is concluded.
\end{proof}

Based on  the convergence rate estimation
 in \eqref{lanlan} of Lemma \ref{Lemma 2.1}, we next estimate the approximation error
$||(I-\mathcal{S}_{\phi}^{N})f||$ in Theorem \ref{Theorem x1}.


\begin{theo}\label{Theorem x1}
Suppose that $X^{s}(\phi;\psi^{1}, \psi^{2},\ldots, \psi^{L})$ and
$X^{-s}(\widetilde{\phi};\widetilde{\psi}^{1},
 \widetilde{\psi}^{2}, \ldots, \widetilde{\psi}^{L})$ form a pair of dual $M$-framelet systems
 for $(H^{s}(\mathbb{R}^{d}),H^{-s}(\mathbb{R}^{d}))$. Moreover,  assume that $\phi\in
 H^{\varsigma}(\mathbb{R}^{d})$,
$\widetilde{\phi}\in H^{-t}(\mathbb{R}^{d})$  and
 $\widetilde{\psi}^{\ell}$ has $\kappa+1$ vanishing moments, where  $0<t<s<\varsigma<\kappa+1$,
$\kappa\in\mathbb{N}_{0}$, and $\ell=1,2, \ldots, L$.
 Then  there exists a positive constant $C(s, \varsigma)$  such that
\begin{align}\begin{array}{lll} \label{bound1}
\displaystyle
||(I-\mathcal{S}_{\phi}^{N})f||_{H^{s}(\mathbb{R}^{d})} \leq C(s,
\varsigma) m^{-\eta_{\kappa+1}(s,  \varsigma)
N}||f||_{H^{\varsigma}(\mathbb{R}^{d})}, \forall f\in
H^{\varsigma}(\mathbb{R}^{d}),
\end{array}
\end{align}
 where  $\eta_{\kappa+1}$ is defined
in \eqref{rt609}.
\end{theo}

\begin{proof}
Denote by $l^{2}(\mathbb{Z}^{d}\times \mathbb{N}_{0}\times
\mathbb{Z}^{d} \times L)$  the space of square summable sequences
supported on $\mathbb{Z}^{d}\times \mathbb{N}_{0}\times
\mathbb{Z}^{d} \times L.$ Let
 $\emph{\textsf{P}}: H^{s}(\mathbb{R}^{d})\rightarrow l^{2}(\mathbb{Z}^{d}\times \mathbb{N}_{0}\times
\mathbb{Z}^{d} \times L)$ be the analysis operator of
$X^{s}(\phi;\psi^{1}, \psi^{2},\ldots, \psi^{L})$. That is, for any
$g\in H^{s}(\mathbb{R}^{d})$,
\begin{align}\nonumber
\emph{\textsf{P}}g:=\Big\{\langle
g,\phi_{0,n}\rangle_{H^{s}(\mathbb{R}^{d})},\langle
g,\psi^{\ell,s}_{j,k} \rangle_{H^{s}(\mathbb{R}^{d})}:  n, k\in
\mathbb{Z}^{d}, j\in \mathbb{N}_{0},\ell=1, \ldots, L\Big\}.
\end{align}
 By \eqref{hanbin3},
$\emph{\textsf{P}}$ is a bounded operator from
$H^{s}(\mathbb{R}^{d})$ to $l^{2}$. Then
\begin{align}\label{fg}
||\emph{\textsf{P}}g||_{2}&\leq
||\emph{\textsf{P}}||||g||_{H^{s}(\mathbb{R}^{d})}.
\end{align}
By the isomorphic map $\theta_{s}:
H^{s}(\mathbb{R}^{d})\longrightarrow H^{-s}(\mathbb{R}^{d})$ defined
by
$$\widehat{\theta_{s}g}(\xi)=\widehat{g}(\xi)(1+||\xi||_{2}^{2})^{s}, \forall g\in H^{s}(\mathbb{R}^{d}),$$
it is easy to prove that \eqref{hanbin3} holds with $g$ being
replaced by any $\widetilde{g}\in H^{-s}(\mathbb{R}^{d}).$
Therefore, by \cite[Theorem 2.1]{Li0},
\begin{align}\label{fgzhuhai}
||\emph{\textsf{P}}||&\leq h(s, \varsigma),
\end{align}
where
$$h(s, \varsigma)=\Big(L||[\widehat{\phi},
\widehat{\phi}]_{\varsigma}||_{L^{\infty}}\big[1+\frac{m^{d}}{(2\pi)^{d}}(\frac{m^{2(\varsigma+s)}2^{s}}{m^{2(\varsigma-s)}-1}+\frac{2^{s}}{1-m^{-2s}})\max_{1\leq\ell\leq
L}\{||\widehat{b^{\ell}}||_{L^{\infty}}\}\big]\Big)^{1/2}.$$ Next we
compute  $\emph{\textsf{P}}^{*}$, the adjoint operator of
$\emph{\textsf{P}}$. For any $c\in l^{2}(\mathbb{Z}^{d}\times
\mathbb{N}_{0}\times \mathbb{Z}^{d} \times L)$ and  $g\in
H^{s}(\mathbb{R}^{d})$,
$$\langle \emph{\textsf{P}}^{*}c, g\rangle_{H^{s}(\mathbb{R}^{d})}=\langle c, \emph{\textsf{P}}g\rangle_{l^{2}}=\sum_{k\in \mathbb{Z}^{d}}c_{k}
\langle g,\phi_{0,k}\rangle_{H^{s}(\mathbb{R}^{d})}+
\sum^{L}_{\ell=1}\sum_{j\in \mathbb{\mathbb{N}}_{0}}\sum_{k\in
\mathbb{Z}^{d}} c^{-s}_{j,k,\ell} \langle g,\psi^{\ell,s}_{j,k}
\rangle_{H^{s}(\mathbb{R}^{d})},$$ where the elements are $c_{k}$
and $c^{-s}_{j,k,\ell}$. Therefore,
$$\emph{\textsf{P}}^{*}c=\sum_{k\in \mathbb{Z}^{d}}c_{k}
\phi_{0,k}+ \sum^{L}_{\ell=1}\sum_{j\in
\mathbb{\mathbb{N}}_{0}}\sum_{k\in \mathbb{Z}^{d}} c^{-s}_{j,k,\ell}
\psi^{\ell,s}_{j,k}.$$ From
$||\emph{\textsf{P}}^{*}||=||\emph{\textsf{P}}||$, we arrive at
\begin{align}\label{dfg}||\emph{\textsf{P}}^{*}(c)||_{H^{s}(\mathbb{R}^{d})}\leq
||\emph{\textsf{P}}|| ||c||_{l^{2}}.\end{align} For any $f\in
H^{\varsigma}(\mathbb{R}^{d})$,  it follows from \eqref{dfg},
\eqref{fgzhuhai} and \eqref{lanlan} that
\begin{align}\label{guji}\begin{array}{lll}
\displaystyle\displaystyle
||\sum^{L}_{\ell=1}\sum^{\infty}_{j=N}\sum_{k\in \mathbb{Z}^{d}}
\langle f, \widetilde{\psi}^{\ell,-s}_{j,k} \rangle
\psi^{\ell,s}_{j,k}||_{H^{s}(\mathbb{R}^{d})}
&\displaystyle\leq||\emph{\textsf{P}}||\Big(
\sum^{L}_{\ell=1}\sum^{\infty}_{j=N}\sum_{k\in \mathbb{Z}^{d}}
\displaystyle|\langle f, \widetilde{\psi}^{\ell,-s}_{j,k} \rangle|^{2}\Big)^{1/2}\\
&\leq h(s, \varsigma) \sqrt{G(s,t)} m^{-N\eta_{\kappa+1}(s,
\varsigma)}||f||_{H^{\varsigma}(\mathbb{R}^{d})},
\end{array}
\end{align}
where
$$G(s,t)=\sum^{L}_{\ell=1}G(\widehat{\widetilde{b}_{\ell}},
s,t). $$  Herein, $\widehat{\widetilde{b}_{\ell}}$ is the mask
symbol of $\widetilde{\psi}^{\ell}$, and
$G(\widehat{\widetilde{b}_{\ell}}, s,t, \widetilde{\phi})$ is
defined via \eqref{gbound}; namely,
\begin{align}\label{gbbse} G(\widehat{\widetilde{b}_{\ell}},
s,t)=\frac{m^{d}||[\widehat{\widetilde{\phi}},\widehat{\widetilde{\phi}}]_{-t}||_{L^{\infty}(\mathbb{R}^{d})}}{(2\pi)^{d}}\Big(\frac{2^{t}C_{0}(\widehat{\widetilde{b}_{\ell}})^{2}m^{-2(\kappa+1)}}{1-m^{-2(\kappa+1-s)}}+C_{1}(\widehat{\widetilde{b}_{\ell}},
s, t)\Big),\end{align} where  $C_{1}(\widehat{\widetilde{b}_{\ell}},
s, t)$ is defined  via \eqref{zjg} by replacing
$\widehat{\widetilde{b}}$ with $\widehat{\widetilde{b}_{\ell}}$. In
\eqref{gbbse}, when $t$ decreases (increases),
$||[\widehat{\widetilde{\phi}},\widehat{\widetilde{\phi}}]_{-t}||_{L^{\infty}(\mathbb{R}^{d})}$
increases (decreases) while $C_{1}(\widehat{\widetilde{b}_{\ell}},
s, t)$ decreases (increases). On other hand,
$C_{1}(\widehat{\widetilde{b}_{\ell}}, s, t)$ is continuous with
respect to $t$, and it is easy to prove by the dominated convergence
theorem  that
$||[\widehat{\widetilde{\phi}},\widehat{\widetilde{\phi}}]_{-t}||_{L^{\infty}(\mathbb{R}^{d})}$
is also continuous with respect to $t$. Therefore, there exists
$t_{0}\in (-\infty, s)$ such that
$$G(s,t_{0})=\min_{t\in (-\infty, s)} G(s,t).$$ Now we select
$$C(s, \varsigma):=h(s, \varsigma) \sqrt{G(s,t_{0})}$$
to conclude the proof.
\end{proof}

\begin{rem} In Theorem \ref{Theorem x1}, $X^{-s}(\widetilde{\phi};\widetilde{\psi}^{1},
 \widetilde{\psi}^{2}, \ldots, \widetilde{\psi}^{L})$ is any
 framelet system in $H^{-t}(\mathbb{R}^{d})$, and it is not
 necessary in $L^{2}(\mathbb{R}^{d})$. Moreover, the error estimate
 given in \eqref{bound1} holds for any  $ f\in
H^{\varsigma}(\mathbb{R}^{d})$, where $ 0<t<s<\varsigma.$ It follows
from \eqref{rt} and Theorem \ref{Theorem x1} that $f$ can be
approximated by using the inner products $\langle f,
\widetilde{\phi}_{0,k}\rangle$ and $\langle f,
\widetilde{\psi}^{\ell,-s}_{j,k} \rangle$. Now  two necessary
procedures are carried out to construct its approximation. The first
step is to construct a refinable function in
$H^{\varsigma}(\mathbb{R}^{d})$, which has the desired sum rules and
Sobolev smoothness. This can be easily accomplished by   box
splines. We refer to \cite{Hanbin1,Riemenschneider} for the Sobolev
smoothness and sum rules of  box splines. On other hand,  we need to
compute the inner products $\langle f,
\widetilde{\phi}_{0,k}\rangle$ and $\langle f,
\widetilde{\psi}^{\ell,-s}_{j,k} \rangle$. For any system
$X^{-s}(\widetilde{\phi};\widetilde{\psi}^{1}, \widetilde{\psi}^{2},
\ldots, \widetilde{\psi}^{L})$, it is difficult to exactly compute
the   inner products. Fortunately, however,  the computational
problem can be solved by a special framelet system from the
refinable function $\Delta$ to be defined in \eqref{caiyanghanshu}.

\end{rem}

\subsection{Shift-perturbed  approximation system in $H^{s}(\mathbb{R}^{d})$}
\label{buguize} Recall that the operator $\mathcal{S}_{\phi}^{N}$ in
\eqref{rt} is defined  via the system $\{\widetilde{\phi}_{0,k},
\phi_{0,k}, \widetilde{\psi}^{\ell,-s}_{j,k}, \psi^{\ell,s}_{j,k}:
k\in \mathbb{Z}^{d}, j=0, 1, \ldots, N-1, \ell=1, 2, \ldots, L\}$.
 We next use  a more
concise system to reexpress  $\mathcal{S}_{\phi}^{N}$. We start with
the construction of dual $M$-framelets in $(H^{s}(\mathbb{R}^{d}),
H^{-s}(\mathbb{R}^{d})).$
 Assume that  $\phi$ and $\widetilde{\phi}$ are the
$M$-refinable functions in $H^{s}(\mathbb{R}^{d})$ and
$H^{-s}(\mathbb{R}^{d})$, respectively.  Moreover,  $\phi$ has
$\kappa+1$
 sum rules with $\kappa\in\mathbb{N}_{0}$; namely, there exists a
$2\pi\mathbb{Z}^{d}$-periodic trigonometric polynomial $\widehat{Y}$
with $\widehat{Y}(0)\neq0$ such that $\widehat{a}$, the mask symbol
of $\phi$,  satisfies
$$
\widehat{Y}(M^{T}\cdot)\widehat{a}(\cdot+2\pi\gamma)=\delta_{\gamma}\widehat{Y}(\cdot)+\mbox{O}(||\cdot||_{2}^{\kappa+1}),
\ \forall\gamma\in\Gamma_{M^{T}},$$ where $\Gamma_{M^{T}}$ is
defined in the sentence above \eqref{yy1}, and $\{\delta_{\gamma}\}$
is a Dirac sequence such that $\delta_{0}=1$ and $\delta_{\gamma}=0$
for any $\gamma\neq0$. By  the mixed extension principle (MEP)
\cite[Algorithm 4.1]{Li0}, we can construct dual framelet systems
$X^{s}(\phi;\psi^{1}, \psi^{2},\ldots, \psi^{m^{d}})$ and
$X^{-s}(\widetilde{\phi};\widetilde{\psi}^{1},
 \widetilde{\psi}^{2}, \ldots, \widetilde{\psi}^{m^{d}})$
such that $\widetilde{\psi}^{1},
 \widetilde{\psi}^{2}, \ldots, \widetilde{\psi}^{m^{d}}$ all have $\kappa+1$ vanishing
moments. The key ingredient  of MEP is to construct the mask symbols
 $\{ \widehat{b^{1}},
\ldots, \widehat{b^{m^{d}}}\}$ of $\{\psi^{1}, \ldots,
\psi^{m^{d}}\}$, and $\{ \widehat{\widetilde{b}^{1}}, \ldots,
\widehat{\widetilde{b}^{m^{d}}}\}$ of $\{\widetilde{\psi}^{1},
\ldots, \widetilde{\psi}^{m^{d}}\}$ such that they satisfy
\begin{align}\label{hanbin9} \sum^{m^{d}}_{\ell=1}\overline{\widehat{b^{\ell}}}(\cdot+\gamma_{j})\widehat{\widetilde{b}^{\ell}}(\cdot)=\delta_{\gamma_{j}}-\overline{\widehat{a}}(\cdot+\gamma_{j})\widehat{\widetilde{a}}(\cdot+\gamma_{j}), \ \forall  j\in\{1, 2, \ldots, m^{d}\},\end{align}
where $\widehat{\widetilde{a}}$ is the mask symbol of
$\widetilde{\phi}$. From \eqref{hanbin9}, we arrive at
\begin{align}\label{rt1} \begin{array}{llllll} \displaystyle \mathcal{S}_{\phi}^{N}f=\displaystyle\sum_{k\in \mathbb{Z}^{d}} \langle
f,\widetilde{\phi}^{-s}_{N,k} \rangle
\phi^{s}_{N,k},\end{array}\end{align} where
\begin{align}\label{dingyi}
\phi^{s}_{N,k}=m^{N(d/2-s)}\phi(M^{N}\cdot-k), \quad
\widetilde{\phi}^{-s}_{N,k}=m^{N(d/2+s)}\widetilde{\phi}(M^{N}\cdot-k).\end{align}
 That is, we can use the system $\{\widetilde{\phi}^{-s}_{N,k},
\phi^{s}_{N,k}\}$ to reexpress  $\mathcal{S}_{\phi}^{N}$. By
\eqref{bound1},  when the scale level $N$ is sufficiently large, $f$
can be approximately reconstructed by using the inner products $
\langle f, \widetilde{\phi}^{-s}_{N,k} \rangle, k\in\mathbb{Z}^{d}$.

Let $\alpha>0.$ By $l^{\alpha}(\mathbb{Z}^{d}) $ we denote the
linear space of all sequence $\theta=\{\theta_{k}\}:
\mathbb{Z}^{d}\rightarrow \mathbb{R}^{d}$ such that
\begin{align}\label{fanshu} ||\theta||_{l^{\alpha}(\mathbb{Z}^{d})}:=\big(\sum_{k\in\mathbb{Z}^{d}}||\theta_{k}||^{\alpha}_{2}\big)^{1/\alpha}<\infty.\end{align}
 For
$\lambda\in\mathbb{R}^{d}$, a sequence
$\varepsilon:=\{\varepsilon_{k}: k\in \mathbb{Z}^{d}\}$ is
 $\lambda$-clustered in $l^{\alpha}(\mathbb{Z}^{d})$  if
\begin{align}\label{ghbv}
||\varepsilon-\lambda||_{l^{\alpha}(\mathbb{Z}^{d})}=\big(\sum_{k\in\mathbb{Z}^{d}}||\varepsilon_{k}-\lambda||_{2}^{\alpha}\big)^{1/\alpha}<\infty.
\end{align}
By \eqref{ghbv}, any $\lambda$-clustered sequence can be decomposed
into a sequence in $l^{\alpha}(\mathbb{Z}^{d})$ and a constant
sequence $\{\lambda\}$.
 For a $\lambda$-clustered sequence $\varepsilon$, define the operator $\mathcal{S}_{\phi;
\varepsilon}^{N}: H^{s}(\mathbb{R}^{d}) \longrightarrow
L^{2}(\mathbb{R}^{d})$  by
\begin{align} \label{suanzidigyi}   \mathcal{S}_{\phi; \varepsilon}^{N}f= \sum_{k\in \mathbb{Z}^{d}} \langle f,
m^{N(d/2+s)}\widetilde{\phi}(M^{N}\cdot-k-\varepsilon_{k})\rangle
\phi^{s}_{N,k}, \forall f\in H^{s}(\mathbb{R}^{d}),\end{align} where
$\phi$ and $\widetilde{\phi}$ are as in \eqref{rt1}.  By the direct
observation, $\mathcal{S}_{\phi; \varepsilon}^{N}$ is derived from
the perturbation of $\mathcal{S}_{\phi}^{N}$ with respect to the
shifts of $\widetilde{\phi}^{-s}_{N,k}, k\in \mathbb{Z}^{d}$. We
shall use $\mathcal{S}_{\phi; \varepsilon}^{N}$ to construct the
nonuniform sampling approximation in Section \ref{bugzecy}. A
crucial  task is to estimate $||(I-\mathcal{S}_{\phi;
\varepsilon}^{N})f||_{2}$. To the best of our knowledge, this
problem has not been solved in the literature. We shall establish
the error estimate  in Theorem \ref{Theorem xe4}. Incidently,  the
estimation of $||(I-\mathcal{S}_{\phi; \varepsilon}^{N})f||_{2}$ to
be given will guarantee  that the range of $\mathcal{S}_{\phi;
\varepsilon}^{N}$ is contained in $L^{2}(\mathbb{R}^{d})$. The
following lemma is useful for proving Theorem \ref{Theorem xe4}.

\begin{lem}\label{Lemma Xr} Let $J\geq\log_{m}d$ and $s>d/2$. Then
\begin{align} \label{op3} \displaystyle\sum_{||j||_{2}\geq m^{J}}||j||_{2}^{-2s}\leq
d^{1+s-d}2^{2s}\Big[\frac{2s-d+1}{2s-d}+\frac{2s}{2s-1}\Big]m^{-J(2s-d)}.\end{align}
\end{lem}

\begin{proof} We intend to give the upper bound of $\sum_{||j||_{1}\geq
m^{J}}||j||_{1}^{-2s}$, and then use the equivalence of the norms of
$\mathbb{R}^{d}$ to prove \eqref{op3}. It is easy to check that
\begin{align} \label{zihe}\big \{j\in \mathbb{Z}^{d}: ||j||_{1}\geq
m^{J}\big \}\subseteq \bigcup^{d}_{k=1}\big \{j=(j_{1}, j_{2},
\ldots, j_{d}): |j_{k}|\geq m^{J}/d, j_{\ell}\in \mathbb{Z},
\ell\neq k\big \}.\end{align} By \eqref{zihe},
\begin{align}\label{guji2}\begin{array}{lllll}
\displaystyle\sum_{||j||_{1}\geq
m^{J}}||j||_{1}^{-2s}&\displaystyle\leq
d\Big[\sum_{|j_{1}|\geq\lfloor m^{J}/d\rfloor}\sum_{j_{2}\in
\mathbb{Z}}\cdots \sum_{j_{d}\in
\mathbb{Z}}\frac{1}{(|j_{1}|+|j_{2}|+\ldots+|j_{d}|)^{2s}} \Big],
\end{array}
\end{align}
where $\lfloor x\rfloor$ denotes the largest integer that is not
larger than $x$.

Noticing  that its sums involved   have  nothing to do with the
signs of the components of $j$,  the upper bound in \eqref{guji2}
can be estimated as follows,
\begin{align}\label{guji3}\begin{array}{lllll}
\displaystyle d\sum_{|j_{1}|\geq\lfloor
m^{J}/d\rfloor}\sum_{j_{2}\in \mathbb{Z}}\cdots \sum_{j_{d}\in
\mathbb{Z}}\frac{1}{(|j_{1}|+|j_{2}|+\ldots+|j_{d}|)^{2s}}\\
\displaystyle\leq d2^{d}\Big[\sum^{\infty}_{j_{1}=\lfloor
m^{J}/d\rfloor}\sum^{\infty}_{j_{2}=1}\cdots
\sum^{\infty}_{j_{d}=1}\frac{1}{(j_{1}+j_{2}+\ldots+j_{d})^{2s}}
+\sum^{\infty}_{j_{1}=\lfloor m^{J}/d\rfloor}\frac{1}{j_{1}^{2s}}
\Big].
\end{array}
\end{align}

 For any $a>0 $, $N\geq1$ and $\imath>1$, it is easy
to check that
\begin{align}\label{guji4}
\sum^{\infty}_{n=N}\frac{1}{(a+n)^{\imath}}\leq\int^{\infty}_{N-1}\frac{1}{(a+x)^{\imath}}dx=\frac{1}{\imath-1}\frac{1}{(a+N-1)^{\imath-1}}.
\end{align}
Applying \eqref{guji4} for $d-1$ times when $N=1$, we obtain
\begin{align}\label{guji5}\begin{array}{lllll}
\displaystyle\sum^{\infty}_{j_{1}=\lfloor
m^{J}/d\rfloor}\sum^{\infty}_{j_{2}=1}\cdots
\sum^{\infty}_{j_{d}=1}\frac{1}{(j_{1}+j_{2}+\ldots+j_{d})^{2s}}
\leq \prod^{d-1}_{l=1}\frac{1}{2s-l}\sum^{\infty}_{j_{1}=\lfloor
m^{J}/d\rfloor}\frac{1}{j_{1}^{2s-d+1}}.
\end{array}
\end{align}
Using \eqref{guji4} again, we have
\begin{align}\label{guji5x}\begin{array}{lllll}
\displaystyle \sum^{\infty}_{j_{1}=\lfloor
m^{J}/d\rfloor}\frac{1}{j_{1}^{2s-d+1}}&\displaystyle=\sum^{\infty}_{j_{1}=\lfloor
m^{J}/d\rfloor+1}\frac{1}{j_{1}^{2s-d+1}}+\frac{1}{\lfloor
m^{J}/d\rfloor^{2s-d+1}}\\
&\displaystyle\leq \frac{1}{\lfloor
m^{J}/d\rfloor^{2s-d}}\big(\frac{1}{2s-d}+\frac{1}{\lfloor
m^{J}/d\rfloor}\big).
\end{array}
\end{align}
Similarly, \begin{align}\label{guji6}\begin{array}{lllll}
\displaystyle \sum^{\infty}_{j_{1}=\lfloor
m^{J}/d\rfloor}\frac{1}{j_{1}^{2s}}\leq \frac{1}{\lfloor
m^{J}/d\rfloor^{2s-1}}\frac{1}{2s-1}+\frac{1}{\lfloor
m^{J}/d\rfloor^{2s}}.
\end{array}
\end{align}
 Combining \eqref{guji2}, \eqref{guji3}, \eqref{guji5},
 \eqref{guji5x} and \eqref{guji6},  we have
\begin{align}\label{guji1}\begin{array}{lllll}
\displaystyle\sum_{||j||_{1}\geq m^{J}}||j||_{1}^{-2s}
&\displaystyle\leq
d2^{d}\Big[\prod^{d-1}_{l=1}\frac{1}{2s-l}\big(\frac{1}{\lfloor
m^{J}/d\rfloor^{2s-d}}\frac{1}{2s-d}+\frac{1}{\lfloor
m^{J}/d\rfloor^{2s-d+1}}\big)\\
& \displaystyle \quad \quad \quad  +\frac{1}{\lfloor
m^{J}/d\rfloor^{2s-1}}\frac{1}{2s-1}+\frac{1}{\lfloor
m^{J}/d\rfloor^{2s}}\Big]\\
&\displaystyle\leq
d2^{d}\Big[\prod^{d-1}_{l=1}\frac{1}{2s-l}\frac{1}{\lfloor
m^{J}/d\rfloor^{2s-d}}\frac{2s-d+1}{2s-d}+\frac{1}{\lfloor
m^{J}/d\rfloor^{2s-1}}\frac{2s}{2s-1}\Big]\\
&\displaystyle\leq d2^{d}\frac{1}{\lfloor
m^{J}/d\rfloor^{2s-d}}\Big[\frac{2s-d+1}{2s-d}+\frac{2s}{2s-1}\Big]\\
&\displaystyle\leq
d^{1+2s-d}2^{2s}\Big[\frac{2s-d+1}{2s-d}+\frac{2s}{2s-1}\Big]m^{-J(2s-d)}.
\end{array}
\end{align}
From
$$||\cdot||_{2}\leq ||\cdot||_{1}\leq\sqrt{d}||\cdot||_{2},$$ we
arrive at
\begin{align} \label{jck}\sum_{||j||_{2}\geq m^{J}}||j||_{2}^{-2s}\leq \sum_{||j||_{1}\geq m^{J}}||j||_{2}^{-2s}\leq d^{-s}\sum_{||j||_{1}\geq m^{J}}||j||_{1}^{-2s}.\end{align}
Now by \eqref{guji1} and \eqref{jck},  the proof of \eqref{op3} can
be concluded.
\end{proof}

Suppose that  $\varepsilon$ is  any $\lambda$-clustered sequence
defined in \eqref{ghbv}. It can be decomposed into a sequence
$\theta$ in $l^{\alpha}(\mathbb{Z}^{d})$ and a constant sequence
$\{\lambda\}$. The procedures for estimating
$||(I-\mathcal{S}_{\phi; \varepsilon}^{N})f||_{2}$ are sketched as
follows. In Theorem \ref{Theorem x2} we  estimate
$||(I-\mathcal{S}_{\phi; \theta}^{N})f||_{2}$ for the perturbation
sequence $\theta\in l^{\alpha}(\mathbb{Z}^{d})$. Then in Lemma
\ref{Theorem x3}, the error
$||\big(I-\mathcal{S}_{\phi;\theta}^{N}\big)(f-f(\cdot+M^{-N}\lambda))||_{2}$
for any $\lambda\in\mathbb{R}^{d}$ is estimated. Having the two
error estimations above, we estimate $||(I-\mathcal{S}_{\phi;
\varepsilon}^{N})f||_{2}$ in Theorem \ref{Theorem xe4}.

\begin{theo}\label{Theorem x2}
Let $\phi\in H^{s}(\mathbb{R}^{d})$ and $\widetilde{\phi}\in
H^{-s}(\mathbb{R}^{d})$ be  both $M$-refinable (where $s>d/2$) such
that
$||\widehat{\widetilde{\phi}}||_{L^{\infty}(\mathbb{R}^{d})}<\infty$.
Suppose that $\theta_{N}:=\{\theta_{N, k}\}_{k\in\mathbb{Z}^{d}}\in
l^{\alpha}(\mathbb{Z}^{d})$, where $0<\alpha<\min\{2s-d, 2\}$. Then
for any $f\in H^{s}(\mathbb{R}^{d})$ and $N\geq
\frac{2s+2-\alpha}{2-\alpha}\log_{m}d$,  there exists a positive
constant $C_{2}(s,\alpha)$ such that
\begin{align}\label{raodong1}\begin{array}{lllll}   ||(I-\mathcal{S}_{\phi; \theta_{N}}^{N})f||_{2}
\leq ||(I-\mathcal{S}_{\phi}^{N})f||_{2}+C_{2}(s,
\alpha)||f||_{H^{s}(\mathbb{R}^{d})}||\theta_{N}||_{\hbox{m}}
m^{-N(\frac{4s+(\alpha-2)d}{2s-\alpha+2}+d)/2},
\end{array}\end{align}
where
\begin{align}\label{fszd} ||\theta_{N}||_{\hbox{m}}:=\max\big\{||\theta_{N}||_{l^{2}(\mathbb{Z}^{d})},
||\theta_{N}||^{\alpha/2}_{l^{\alpha}(\mathbb{Z}^{d})}\big\}\end{align}
with $||\theta_{N}||_{l^{2}(\mathbb{Z}^{d})}$  defined via
\eqref{fanshu} with $\alpha$ being replaced by $2$.
\end{theo}
\begin{proof}
By the triangle inequality, we just need to find a positive constant
$C_{2}(s,\alpha)$ such that
\begin{align}\label{xiangdui05}\begin{array}{lllll}   ||(\mathcal{S}_{\phi}^{N}-\mathcal{S}_{\phi; \theta_{N}}^{N})f||_{2}\leq C_{2}(s, \alpha)||f||_{H^{s}(\mathbb{R}^{d})}
||\theta_{N}||_{\hbox{m}}m^{-N(\frac{4s+(\alpha-2)d}{2s-\alpha+2}+d)/2}.
\end{array}\end{align}
By direct computation, we get
\begin{align}\begin{array}{lllll} \label{split}
\displaystyle\big|\langle f, m^{Nd/2}\widetilde{\phi}(M^{N}\cdot-k)
-m^{Nd/2}\widetilde{\phi}(M^{N}\cdot-k-\theta_{N, k})
\rangle\big|^{2}\\
=\displaystyle \frac{m^{-Nd}}{(2\pi)^{2d}} \Big|\int_{\mathbb{R}^{d}}\widehat{f}(\xi)\overline{\widehat{\widetilde{\phi}}\big((M^{T})^{-N}\xi\big)}e^{i(M^{T})^{-N}k\cdot\xi}(1-e^{i(M^{T})^{-N}\theta_{N, k}\cdot\xi})d\xi\Big|^{2}\\
=\displaystyle \frac{m^{-Nd}}{(2\pi)^{2d}} \Big|\int_{\mathbb{R}^{d}}\widehat{f}(\xi)(1+||\xi||_{2}^{2})^{s/2}\overline{\widehat{\widetilde{\phi}}\big((M^{T})^{-N}\xi\big)}(1+||\xi||_{2}^{2})^{-s/2}e^{i(M^{T})^{-N}k\cdot\xi}(1-e^{i(M^{T})^{-N}\theta_{N, k}\cdot\xi})d\xi\Big|^{2}\\
\leq
\displaystyle\frac{m^{-Nd}}{(2\pi)^{d}}||f||^{2}_{H^{s}(\mathbb{R}^{d})}||\widehat{\widetilde{\phi}}||^{2}_{L^{\infty}(\mathbb{R}^{d})}
\displaystyle\int_{\mathbb{R}^{d}}(1+||\xi||_{2}^{2})^{-s}|1-e^{i(M^{T})^{-N}\theta_{N, k}\cdot\xi}|^{2}d\xi\\
=\displaystyle
\frac{m^{-Nd}}{(2\pi)^{d}}||f||^{2}_{H^{s}(\mathbb{R}^{d})}||\widehat{\widetilde{\phi}}||^{2}_{L^{\infty}(\mathbb{R}^{d})}\big(I_{1}(J)+I_{2}(J)\big),\end{array}
\end{align}
where $ I_{1}(J)=\sum_{||j||_{2}\geq
m^{J}}\int_{\mathbb{T}^{d}}(1+||\xi+2j\pi||_{2}^{2})^{-s}|1-e^{i(M^{T})^{-N}\theta_{N,k}\cdot(\xi+2j\pi)}|^{2}d\xi,
$ and $ I_{2}(J)=\sum_{||j||_{2}<
m^{J}}\int_{\mathbb{T}^{d}}(1+||\xi+2j\pi||_{2}^{2})^{-s}|1-e^{i(M^{T})^{-N}\theta_{N,
k}\cdot(\xi+2j\pi)}|^{2}d\xi $ with $\mathbb{T}^{d}:=[0,2\pi)^{d}$,
and $J (>0)$ to be optimally selected.
 The
two quantities $I_{1}(J)$ and $I_{2}(J)$ are estimated as follows,
\begin{align}\label{cet6}\begin{array}{lllll}
\displaystyle I_{1}(J)&=\displaystyle\sum_{||j||_{2}\geq
m^{J}}\int_{\mathbb{T}^{d}}(1+||\xi+2j\pi||_{2}^{2})^{-s}|1-e^{i(M^{T})^{-N}\theta_{N, k}\cdot(\xi+2j\pi)}|^{2}d\xi\\
&=\displaystyle4\sum_{||j||_{2}\geq
m^{J}}\int_{\mathbb{T}^{d}}(1+||\xi+2j\pi||_{2}^{2})^{-s}\big|\sin\big((M^{T})^{-N}\theta_{N, k}\cdot(\xi+2j\pi)/2\big)\big|^{2}d\xi\\
&\leq \displaystyle4\sum_{||j||_{2}\geq
m^{J}}\int_{\mathbb{T}^{d}}(1+||\xi+2j\pi||_{2}^{2})^{-s}\big|\sin\big((M^{T})^{-N}\theta_{N, k}\cdot(\xi+2j\pi)/2\big)\big|^{\alpha}d\xi\\
&\leq \displaystyle4||(M^{T})^{-N}\theta_{N,
k}||^{\alpha}_{2}\sum_{||j||_{2}\geq
m^{J}}\int_{\mathbb{T}^{d}}(1+||\xi+2j\pi||_{2}^{2})^{-s}||(\xi+2j\pi)/2||^{\alpha}_{2}d\xi\\
&\leq \displaystyle 4||(M^{T})^{-N}\theta_{N,
k}||^{\alpha}_{2}\pi^{\alpha}\sum_{||j||_{2}\geq
m^{J}}(\sqrt{d}+||j||_{2})^{\alpha}\int_{\mathbb{T}^{d}}(1+||\xi+2j\pi||_{2}^{2})^{-s}d\xi\\
&\leq \displaystyle 4||(M^{T})^{-N}\theta_{N,
k}||^{\alpha}_{2}\pi^{\alpha}(2\pi)^{d}\sum_{||j||_{2}\geq
m^{J}}(\sqrt{d}+||j||_{2})^{\alpha}\big[1+(2\pi)^{2}(||j||_{2}-\sqrt{d})^{2}\big]^{-s}\\
&\leq \displaystyle 4||(M^{T})^{-N}\theta_{N,
k}||^{\alpha}_{2}\pi^{\alpha}(2\pi)^{d-2s}2^{2s+\alpha}\sum_{||j||_{2}\geq
m^{J}}||j||^{-2(s-\alpha/2)}_{2}\\
&\leq \displaystyle
(2\pi)^{\alpha+d-2s}2^{4s+2}d^{1+s-d}\Big[\frac{2s-d+1}{2s-d}+\frac{2s}{2s-1}\Big]||\theta_{N,
k}||^{\alpha}_{2}m^{-[J(2s-\alpha-d)+N\alpha]},
\end{array}
\end{align}
where the second inequality is derived  from $\alpha\leq2$, and the
last one from \eqref{op3}. The  quantity $I_{2}$ is estimated as
follows,
\begin{align}\label{di2bufen}\begin{array}{lllll}
\displaystyle I_{2}(J)&\leq\displaystyle\sum_{||j||_{2}<
m^{J}}\int_{\mathbb{T}^{d}}(1+||\xi+2j\pi||_{2}^{2})^{-s}|1-e^{i(M^{T})^{-N}\theta_{N, k}\cdot(\xi+2j\pi)}|^{2}d\xi\\
&\leq \displaystyle (2\pi)^{d}\sum_{||j||_{2}<
m^{J}}\max_{\xi\in [0,2\pi]^{d}}|1-e^{i(M^{T})^{-N}\theta_{N, k}\cdot(\xi+2j\pi)}|^{2}\\
& \leq 4||\theta_{N, k}||_{2}^{2}(2\pi)^{2d+2}m^{-2N+(2+d)J}.
\end{array}
\end{align}
 That is, $I_{1}(J)=\mbox{O}(m^{-[J(2s-\alpha-d)+N\alpha]})$ and
$I_{2}(J)=\mbox{O}(m^{-2N+(2+d)J})$. Therefore,
\begin{align}\label{laibo}
I_{1}(J)+I_{2}(J)=\mbox{O}\big(m^{-\min\{J(2s-\alpha-d)+N\alpha, \
2N-(2+d)J\}}\big).
\end{align}
It is easy to check that if choosing
$J=\frac{2-\alpha}{2s+2-\alpha}N$, then the approximation order in
\eqref{laibo} is optimal. Incidentally, by Lemma \ref{Lemma Xr}, the
condition for the last inequality of \eqref{cet6} is $m^{J}\geq d$.
Therefore, by  $N\geq \frac{2s+2-\alpha}{2-\alpha}\log_{m}d$, the
choice for $J=\frac{2-\alpha}{2s+2-\alpha}N$ is feasible. Now for
this choice, we have
\begin{align}\label{laibo1}
I_{1}(J)+I_{2}(J)=\mbox{O}\big(m^{-N\frac{4s+(\alpha-2)d}{2s-\alpha+2}}\big).
\end{align}
Summarizing \eqref{split}, \eqref{cet6},
 \eqref{di2bufen} and \eqref{laibo1}, we obtain
 \begin{align}\label{yut}\begin{array}{lllll}\displaystyle
|\langle f, m^{Nd/2}\widetilde{\phi}(M^{N}\cdot-k)
-m^{Nd/2}\widetilde{\phi}(M^{N}\cdot-k-\theta_{N, k})\rangle|^{2}\\
\displaystyle\leq
\frac{m^{-Nd}}{(2\pi)^{d-2}}||f||^{2}_{H^{s}(\mathbb{R}^{d})}||\widehat{\widetilde{\phi}}||^{2}_{L^{\infty}(\mathbb{R}^{d})}\Big(
C_{3}(s,\alpha)||\theta_{N,
k}||^{\alpha}_{2}+4(2\pi)^{2d}||\theta_{N,
k}||_{2}^{2}\Big)m^{-N\frac{4s+(\alpha-2)d}{2s-\alpha+2}},
\end{array}\end{align}
where
$$C_{3}(s,\alpha)=(2\pi)^{\alpha-2+d-2s}2^{4s+2}d^{1+s-d}\Big[\frac{2s-d+1}{2s-d}+\frac{2s}{2s-1}\Big].$$
 On the other hand, for any sequence $\{C_{k}\}\in
l^{2}(\mathbb{Z}^{d})$, we have
\begin{align}\begin{array}{lllll} \label{cet7} \displaystyle
\displaystyle||\sum_{k\in\mathbb{Z}^{d}}C_{k}\phi(\cdot-k)||^{2}_{2}&=(2\pi)^{-d}\displaystyle\int_{\mathbb{R}^{d}}|\sum_{k\in\mathbb{Z}^{d}}C_{k}e^{ik\cdot\xi}|^{2}|\widehat{\phi}(\xi)|^{2}d\xi\\
&=\displaystyle\displaystyle(2\pi)^{-d}\int_{\mathbb{T}^{d}}|\sum_{k\in\mathbb{Z}^{d}}C_{k}e^{ik\xi}|^{2}\sum_{\ell\in\mathbb{Z}^{d}}|\widehat{\phi}(\xi+2\ell\pi)|^{2}d\xi\\
&\displaystyle\leq||[\widehat{\phi},\widehat{\phi}]_{0}||_{L^{\infty}(\mathbb{T}^{d})}\sum_{k\in\mathbb{Z}^{d}}|C_{k}|^{2},
\end{array}
\end{align}
where the bracket product
$||[\widehat{\phi},\widehat{\phi}]_{0}||_{L^{\infty}(\mathbb{T}^{d})}$
is defined in \eqref{opq}.
 Then from \eqref{cet7}  and  \eqref{yut}, we arrive at
\begin{align}\label{ghuilai}\begin{array}{lllll} ||(\mathcal{S}_{\phi}^{N}-\mathcal{S}_{\phi; \theta_{N}}^{N})f||^{2}_{2}\\
=\displaystyle||\sum_{k\in \mathbb{Z}^{d}} \langle f,
m^{N(d/2+s)}\widetilde{\phi}(M^{N}\cdot-k)\rangle
\phi^{s}_{N,k}-\sum_{k\in \mathbb{Z}^{d}} \langle f,
m^{N(d/2-s)}\widetilde{\phi}(M^{N}\cdot-k-\theta_{N, k})\rangle
\phi^{s}_{N,k}||_{2}^{2} \\
 =\displaystyle||\sum_{k\in \mathbb{Z}^{d}} \langle f,
m^{Nd/2}\widetilde{\phi}(M^{N}\cdot-k)-m^{Nd/2}\widetilde{\phi}(M^{N}\cdot-k-\theta_{N,
k})\rangle
\phi_{N,k}||_{2}^{2} \\
 \displaystyle\leq
||[\widehat{\phi},\widehat{\phi}]_{0}||_{L^{\infty}(\mathbb{T}^{d})}\sum_{k\in
\mathbb{Z}^{d}}|\langle f,m^{Nd/2}\widetilde{\phi}(M^{N}\cdot-k)
-m^{Nd/2}\widetilde{\phi}(M^{N}\cdot-k-\theta_{N, k}) \rangle|^{2}\\
\leq
\displaystyle\frac{m^{-Nd}}{(2\pi)^{d-2}}||f||^{2}_{H^{s}(\mathbb{R}^{d})}||\widehat{\widetilde{\phi}}||^{2}_{L^{\infty}(\mathbb{R}^{d})}||[\widehat{\phi},\widehat{\phi}]_{0}||_{L^{\infty}(\mathbb{T}^{d})}\Big(
C_{3}(s,\alpha)||\theta_{N}||^{\alpha}_{l^{\alpha}(\mathbb{Z}^{d})}+||\theta_{N}||_{l^{2}(\mathbb{Z}^{d})}^{2}\Big)m^{-N\frac{4s+(\alpha-2)d}{(2s-\alpha+2)}}\\
\leq
\displaystyle\frac{||f||^{2}_{H^{s}(\mathbb{R}^{d})}}{(2\pi)^{d-2}}||\widehat{\widetilde{\phi}}||^{2}_{L^{\infty}(\mathbb{R}^{d})}||[\widehat{\phi},\widehat{\phi}]_{0}||_{L^{\infty}(\mathbb{T}^{d})}\Big(
C_{3}(s,\alpha)+4(2\pi)^{2d}\Big)
||\theta_{N}||_{\hbox{m}}m^{-N[\frac{4s+(\alpha-2)d}{2s-\alpha+2}+d]},
\end{array}\end{align}
where $\phi_{N,k}=m^{Nd/2}\phi(M^{N}\cdot-k),$ and
$||\theta_{N}||_{\hbox{m}}$ is defined in \eqref{fszd}.
 Now we select
\begin{align}\label{C2} C_{2}(s, \alpha):=||\widehat{\widetilde{\phi}}||_{L^{\infty}(\mathbb{R}^{d})}\sqrt{\frac{||[\widehat{\phi},\widehat{\phi}]_{0}||_{L^{\infty}(\mathbb{T}^{d})}}{(2\pi)^{d-2}}\Big(
C_{3}(s,\alpha)+4(2\pi)^{2d}\Big)}\end{align} to conclude the proof
of \eqref{xiangdui05}.
\end{proof}

\begin{rem}\label{lubax}
(I) By the perturbation estimate  in Theorem \ref{Theorem x2}
\eqref{raodong1}, the approximation $\mathcal{S}_{\phi}^{N}f$ of $f$
is robust to the perturbation sequence $\theta_{N}$. Moreover, if
\begin{align}\nonumber \max\big(||\theta_{N}||^{\alpha}_{l^{\alpha}(\mathbb{Z}^{d})},
||\theta_{N}||_{l^{2}(\mathbb{Z}^{d})}^{2}\big)=\hbox{o}(m^{N\gamma})\end{align}
where $\gamma<\frac{4s+(\alpha-2)d}{2s-\alpha+2}+d$, then
$\displaystyle\lim_{N\rightarrow\infty}\mathcal{S}_{\phi,
\theta_{N}}^{N}f=f$ in the sense of $||\cdot||_{2}$. In other words,
as $N$ increases, so does the capability for anti-perturbation of
$\mathcal{S}_{\phi}^{N}$.

(II) The  quantity $I_{1}(J)$ can be bounded in Theorem \ref{Theorem
x2} \eqref{cet6} provided that $s>d/2.$ In this sense, the
robustness of $\mathcal {S}_{\phi}^{N}f$ to perturbation is closely
related  to the condition $s>d/2$. As mentioned in Section
\ref{section1}, the condition will be  crucial for our construction
of nonuniform sampling approximation.

 (III)  By
Theorem \ref{Theorem x1} \eqref{bound1},
$\displaystyle\lim_{N\rightarrow
\infty}||f-\mathcal{S}_{\phi}^{N}f||_{H^{s}(\mathbb{R}^{d})}=0$.
However, due to
$$\displaystyle\lim_{N\rightarrow\infty}||\phi_{N,k}||_{H^{s}(\mathbb{R}^{d})}=\displaystyle\lim_{N\rightarrow\infty}||m^{Nd/2}\phi(M^{N}\cdot-k)||_{H^{s}(\mathbb{R}^{d})}=+\infty,$$
the conditions in Theorem \ref{Theorem x2}  can not guarantee that
$\displaystyle\lim_{N\rightarrow
\infty}||(\mathcal{S}_{\phi}^{N}-\mathcal{S}_{\phi;
\theta_{N}}^{N})f||_{H^{s}(\mathbb{R}^{d})}=0$ nor
$\displaystyle\lim_{N\rightarrow \infty}||f-\mathcal{S}_{\phi;
\theta_{N}}^{N}f||_{H^{s}(\mathbb{R}^{d})}=0$.

\end{rem}


\begin{lem}\label{Theorem x3} Let $s>d/2$.  The sequence
$\theta_{N}$ belongs to $ l^{\alpha}(\mathbb{Z}^{d})$, where
$0<\alpha<\min\{2s-d, 2\}$. Suppose that the two $M$-refinable
functions  $\widetilde{\phi}\in H^{-s}(\mathbb{R}^{d})$ and $\phi\in
H^{s}(\mathbb{R}^{d})$ are  as in Theorem \ref{Theorem x1}.
Moreover, $\widetilde{\phi}\in H^{-t}(\mathbb{R}^{d})$ and $\phi\in
H^{\varsigma}(\mathbb{R}^{d})$, where $d/2<t<s<\varsigma$. Assume
that $N\geq \frac{2s+2-\alpha}{2-\alpha}\log_{m}d$ is arbitrary.
Then there exists $\widetilde{C}_{2}
>0$ (being independent of $N$) such that for every $ f\in
H^{\varsigma}(\mathbb{R}^{d})$ and $\lambda_{N} \in \mathbb{R}^{d}$,
it holds
\begin{align}
\begin{array}{lllll} \label{ghgvb}
||\big(I-\mathcal{S}_{\phi;\theta_{N}}^{N}\big)(f-f(\cdot+M^{-N}\lambda_{N}))||_{2}
\leq
\widetilde{C}_{2}||f||_{H^{\varsigma}(\mathbb{R}^{d})}(1+||\theta_{N}||_{\hbox{m}})||\lambda_{N}||_{2}^{\zeta}m^{-N\zeta},
\end{array}\end{align}
where
$||\theta_{N}||_{\hbox{m}}=\max\{||\theta_{N}||_{l^{2}(\mathbb{Z}^{d})},
||\theta_{N}||^{\alpha/2}_{l^{\alpha}(\mathbb{Z}^{d})}\}$  and $
\zeta=\min\big\{\varsigma-s,
1,(\frac{4s+(\alpha-2)d}{2s-\alpha+2}+d)/2 \big\} $.
\end{lem}



\begin{proof}
By the triangle inequality and Theorem \ref{Theorem x2}
\eqref{xiangdui05},  we estimate
$||\big(I-\mathcal{S}_{\phi;\theta_{N}}^{N}\big)\big(f-f(\cdot+M^{-N}\lambda_{N})\big)||_{2}$
as follows,
\begin{align} \begin{array}{lllll} \label{raodong} \displaystyle ||\big(I-\mathcal{S}_{\phi;\theta_{N}}^{N}\big)\big(f-f(\cdot+M^{-N}\lambda_{N})\big)||_{2}\\
 \leq
 ||\big(I-\mathcal{S}_{\phi}^{N}\big)\big(f-f(\cdot+M^{-N}\lambda_{N})\big)||_{2}+||\big(\mathcal{S}_{\phi;\theta_{N}}^{N}-\mathcal{S}_{\phi}^{N}\big)\big(f-f(\cdot+M^{-N}\lambda_{N})\big)||_{2}\\
 \leq
 ||\big(I-\mathcal{S}_{\phi}^{N}\big)\big(f-f(\cdot+M^{-N}\lambda_{N})\big)||_{2}\\
 \ +C_{2}(s,
\alpha) ||\theta_{N}||_{\hbox{m}}m^{-N(\frac{4s+(\alpha-2)d}{2s-\alpha+2}+d)/2}||f-f(\cdot+M^{-N}\lambda_{N})||_{H^{s}(\mathbb{R}^{d})}.
\end{array}\end{align}
Invoking  \eqref{bound1}, we get
\begin{align}\label{zjxc} \begin{array}{lllll} ||\big(I-\mathcal{S}_{\phi}^{N}\big)\big(f-f(\cdot+M^{-N}\lambda_{N})\big)||_{2}&\leq C\big(\frac{t+s}{2},
s\big) m^{-\eta_{\kappa+1}(\frac{t+s}{2},  s)
N}||f-f(\cdot+M^{-N}\lambda_{N})||_{H^{s}(\mathbb{R}^{d})}.\end{array}\end{align}
 On the other hand,
\begin{align}\label{dddf}\begin{array}{lllll}\displaystyle
||f-f(\cdot+M^{-N}\lambda_{N})||_{H^{s}(\mathbb{R}^{d})}\\
=\displaystyle\Big[\frac{1}{(2\pi)^{d}}\int_{\mathbb{R}^{d}}|\widehat{f}(\xi)(1-e^{i(M^{T})^{-N}\lambda_{N}\cdot\xi})|^{2}
(1+||\xi||^{2}_{2})^{s}d\xi\Big]^{1/2}\\
\leq\displaystyle\Big[\frac{1}{(2\pi)^{d}}\int_{\mathbb{R}^{d}}4|\sin\big((M^{T})^{-N}\lambda_{N}\cdot\xi/2\big)|^{2\zeta}
|\widehat{f}(\xi)|^{2}(1+||\xi||^{2}_{2})^{s}d\xi\Big]^{1/2}\\
\leq\displaystyle\Big[\frac{2^{2-2\zeta}||(M^{T})^{-N}\lambda_{N}||^{2\zeta}_{2}}{(2\pi)^{d}}\int_{\mathbb{R}^{d}}||\xi||_{2}^{2\zeta}
|\widehat{f}(\xi)|^{2}(1+||\xi||^{2}_{2})^{s}d\xi\Big]^{1/2}\\
\leq\displaystyle2^{1-\zeta}m^{-N\zeta}||\lambda_{N}||_{2}^{\zeta}||f||_{H^{\varsigma}(\mathbb{R}^{d})}.
\end{array}\end{align}
Select
\begin{align}\nonumber
\widetilde{C}_{2}:=2^{1-\zeta}\max\big\{C\big(\frac{t+s}{2}, s\big),
\ \ C_{2}(s, \alpha) \big\}.
\end{align}
 Now the proof of \eqref{ghgvb} can be  concluded by
\eqref{raodong}, \eqref{zjxc} and \eqref{dddf}.
\end{proof}

The estimate of $||(I-\mathcal{S}_{\phi; \theta_{N}}^{N})f||_{2}$ in
Theorem \ref{Theorem x2} \eqref{raodong1} holds for the perturbation
sequence  $\theta_{N}=\{\theta_{N,k}\}\in
l^{\alpha}(\mathbb{Z}^{d})$. Now based on Theorem \ref{Theorem x2}
and Lemma \ref{Theorem x3}, we estimate $||(I-\mathcal{S}_{\phi;
\varepsilon_{N}}^{N})f||_{2}$ for any $\lambda$-clustered sequence
$\varepsilon_{N}=\{\varepsilon_{N,k}:=\theta_{N,k}+\lambda_{N}\}_{k\in\mathbb{Z}^{d}}$
defined in \eqref{ghbv}.

\begin{theo}\label{Theorem xe4} Let $s>d/2$. Suppose that $N\geq \frac{2s+2-\alpha}{2-\alpha}\log_{m}d$ is arbitrary, and a sequence
$\varepsilon_{N}=\{\varepsilon_{N,k}:=\theta_{N,k}+\lambda_{N}\}_{k\in\mathbb{Z}^{d}}$
is $\lambda_{N}$-clustered in $l^{\alpha}(\mathbb{Z}^{d})$, where
$\lambda_{N}\in\mathbb{R}^{d}$ and $0<\alpha<\min\{2s-d, 2\}$. The
two $M$-refinable functions $\phi\in H^{\varsigma}(\mathbb{R}^{d})$
and $\widetilde{\phi}\in H^{-t}(\mathbb{R}^{d})$ are as in Lemma
\ref{Theorem x3}, where $d/2<t<s<\varsigma$.
 Then there exists $C_{3}>0$ (being independent of  $N$)  such
 that
  \begin{align} \begin{array}{lllll} \label{jyy}
 ||(I-\mathcal{S}_{\phi;\varepsilon_{N}}^{N})f||_{2}\leq||(I-\mathcal{S}_{\phi}^{N})f||_{2}
+
C_{3}||f||_{H^{\varsigma}(\mathbb{R}^{d})}m^{-N\zeta}\big[(1+||\lambda_{N}||_{2}^{\zeta})
||\theta_{N}||_{\hbox{m}}+||\lambda_{N}||_{2}^{\zeta} \big]  \
\end{array}\end{align}
holds  for every $ f\in H^{\varsigma}(\mathbb{R}^{d}),$  where $
\zeta=\min\{\varsigma-s, 1,(\frac{4s+(\alpha-2)d}{2s-\alpha+2}+d)/2
\} $, and
$||\theta_{N}||_{\hbox{m}}=\max\{||\theta_{N}||_{l^{2}(\mathbb{Z}^{d})},
||\theta_{N}||^{\alpha/2}_{l^{\alpha}(\mathbb{Z}^{d})}\}$.
\end{theo}

\begin{proof}
By the Plancherel's theorem, we get
\begin{align}\begin{array}{lllll} \label{split1}
\displaystyle\big\langle f,
m^{Nd/2}\widetilde{\phi}(M^{N}\cdot-k-\theta_{N, k})
-m^{Nd/2}\widetilde{\phi}(M^{N}\cdot-k-\theta_{N,k}-\lambda_{N})\big\rangle\\
=\displaystyle\frac{m^{-Nd/2}}{(2\pi)^{d}}\displaystyle\int_{\mathbb{R}^{d}}\widehat{f}(\xi)(1-e^{i(M^{T})^{-N}\lambda_{N}\cdot\xi})\overline{\widehat{\widetilde{\phi}}((M^{T})^{-N}\xi)}e^{i(M^{T})^{-N}(k+\theta_{N,k})\xi}d\xi\\
=\displaystyle\big\langle f-f(\cdot+M^{-N}\lambda_{N}),
m^{Nd/2}\widetilde{\phi}(M^{N}\cdot-k-\theta_{N,k})\big\rangle.\end{array}
\end{align}
Using the triangle inequality and \eqref{split1}, the error $
||(I-\mathcal{S}_{\phi;\varepsilon_{N}}^{N})f||_{2}$ is estimated as
follows,
\begin{align} \begin{array}{lllll} \label{raodong011}
 ||(I-\mathcal{S}_{\phi;\varepsilon_{N}}^{N})f||_{2}\\
\leq ||(I-\mathcal{S}_{\phi;\theta_{N}}^{N})f||_{2}+
||\mathcal{S}_{\phi;\theta_{N}}^{N}(f-f(\cdot+M^{-N}\lambda_{N}))||_{2}\\
\leq
||(I-\mathcal{S}_{\phi;\theta_{N}}^{N})f||_{2}+||f-f(\cdot+M^{-N}\lambda_{N})||_{2}+
||(I-\mathcal{S}_{\phi;\theta_{N}}^{N})(f-f(\cdot+M^{-N}\lambda_{N}))||_{2}.
\end{array}\end{align}
It follows from Lemma \ref{Theorem x3} \eqref{ghgvb} and
\eqref{dddf} that
\begin{align}
\begin{array}{lllll} \label{gbgb}
||(I-\mathcal{S}_{\phi;\theta_{N}}^{N})(f-f(\cdot+M^{-N}\lambda_{N}))||_{2}+||f-f(\cdot+M^{-N}\lambda_{N})||_{2}\\
\leq
\widetilde{C}_{2}||f||_{H^{\varsigma}(\mathbb{R}^{d})}(1+||\theta_{N}||_{\hbox{m}})||\lambda_{N}||_{2}^{\zeta}m^{-N\widetilde{\eta}_{\kappa+1}}+2^{1-\zeta}m^{-N\zeta}||\lambda_{N}||_{2}^{\zeta}||f||_{H^{\varsigma}(\mathbb{R}^{d})}\\
\leq\widetilde{C}_{2}||f||_{H^{\varsigma}(\mathbb{R}^{d})}(1+||\theta_{N}||_{\hbox{m}})||\lambda_{N}||_{2}^{\zeta}m^{-N\zeta}+2^{1-\zeta}m^{-N\zeta}||\lambda_{N}||_{2}^{\zeta}||f||_{H^{\varsigma}(\mathbb{R}^{d})}.
\end{array}\end{align}
It follows from  $\zeta\leq(\frac{4s+(\alpha-2)d}{2s-\alpha+2}+d)/2$
and Theorem \ref{Theorem x2} \eqref{raodong1} that
\begin{align} \begin{array}{lllll} \label{huaz}
||(I-\mathcal{S}_{\phi;\theta_{N}}^{N})f||_{2}& \leq
||(I-\mathcal{S}_{\phi}^{N})f||_{2}+C_{2}(s,
\alpha)||f||_{H^{s}(\mathbb{R}^{d})}
||\theta_{N}||_{\hbox{m}}m^{-N\zeta}.
\end{array}\end{align}
From \eqref{gbgb} and \eqref{huaz}, we arrive at
  \begin{align} \begin{array}{lllll} \label{zhgj}
 ||\big(I-\mathcal{S}_{\phi;\varepsilon_{N}}^{N}\big)f||_{2}-||(I-\mathcal{S}_{\phi}^{N})f||_{2}\\
\leq
\widetilde{C}_{2}\big(1+||\theta_{N}||_{\hbox{m}}\big)||f||_{H^{\varsigma}(\mathbb{R}^{d})}||\lambda_{N}||_{2}^{\zeta}m^{-N\zeta}+2^{1-\zeta}||f||_{H^{\varsigma}(\mathbb{R}^{d})}m^{-N\zeta}||\lambda_{N}||_{2}^{\zeta}\\
\  \ +C_{2}(s, \alpha)
||\theta_{N}||_{\hbox{m}}||f||_{H^{\varsigma}(\mathbb{R}^{d})}m^{-N\zeta}\\
=||f||_{H^{\varsigma}(\mathbb{R}^{d})}m^{-N\zeta}\big[\widetilde{C}_{2}(1+||\theta_{N}||_{\hbox{m}})||\lambda_{N}||_{2}^{\zeta}+2^{1-\zeta}||\lambda_{N}||_{2}^{\zeta}+C_{2}(s,
\alpha) ||\theta_{N}||_{\hbox{m}}\big].
\end{array}\end{align}
Define
$$C_{3}:=2\max\big\{\widetilde{C}_{2}+2^{1-\zeta}, \ C_{2}(s,
\alpha)\big\}$$ to conclude the proof.
\end{proof}

\begin{rem} \label{kjcw}
(I) It is straightforward to see that  for any
$\lambda_{N}$-clustered sequence $\varepsilon_{N}$ with
$\lambda_{N}\neq0$,
$||\varepsilon_{N}||_{\ell^{\beta}(\mathbb{Z}^{d})}=\infty$ where
$\beta>0$ is arbitrary. Therefore  the estimate of
$||(I-\mathcal{S}_{\phi;\varepsilon_{N}}^{N})f||_{2}$ can not be
given  only   by Theorem \ref{Theorem x2}. Instead, separating the
constant sequence $\{\lambda_{N}\}$ form $\varepsilon_{N}$, we
combine Theorem \ref{Theorem x2} and Lemma \ref{Theorem x3} to
complete the error estimate in Theorem \ref{Theorem xe4}.

(II) For every scale level $N$, it follows from Theorem \ref{Theorem
xe4} \eqref{jyy} that if  \begin{align} \label{ganrtj}
||\lambda_{N}||_{2}^{\zeta}+
\max\{||\theta_{N}||_{l^{2}(\mathbb{Z}^{d})},
||\theta_{N}||^{\alpha/2}_{l^{\alpha}(\mathbb{Z}^{d})}\}
=\hbox{o}(m^{N\gamma}),\end{align} where $\gamma<\zeta/2$, then
$\lim_{N\rightarrow\infty} \mathcal{S}_{\phi;
\varepsilon_{N}}^{N}f=f$ in the sense of $||\cdot||_{2}$. That is,
when the perturbation sequence
$\varepsilon_{N}=\{\varepsilon_{N,k}:=\theta_{N,k}+\lambda_{N}\}_{k\in\mathbb{Z}^{d}}$
is bounded by \eqref{ganrtj}, then the approximation
$\mathcal{S}_{\phi; \varepsilon_{N}}^{N}f$ is robust to the
perturbation. Moreover, for larger scale level $N$,
$\mathcal{S}_{\phi}^{N}f$ performs better against perturbation.
\end{rem}

%
%
%

\section{Approximation to functions in Sobolev spaces by nonuniform samples}
\label{bugzecy}

With the help of Theorem \ref{Theorem x1} and  Theorem  \ref{Theorem
xe4} we establish the following nonuniform sampling theorem, which
states that  any function in $H^{s}(\mathbb{R}^{d})$ (where $s>d/2$)
can be stably reconstructed by nonuniform samples  with a carefully
selected pair of framelets for  $(H^{s}(\mathbb{R}^{d}),
H^{-s}(\mathbb{R}^{d}))$.

\begin{theo}\label{Theorem x4}  Let   $\phi\in H^{s}(\mathbb{R}^{d})$ be  $M$-refinable where $s>d/2$.
Suppose that  $N\geq \frac{2s+2-\alpha}{2-\alpha}\log_{m}d$ is
arbitrary. Assume that $\phi$ belongs to $
 H^{\varsigma}(\mathbb{R}^{d})$ and  has $\kappa+1$
 sum rules where $s<\varsigma<\kappa+1$, and a sequence
$\varepsilon_{N}=\{\varepsilon_{N,k}:=\theta_{N,k}+\lambda_{N}\}_{k\in\mathbb{Z}^{d}}$
is $\lambda_{N}$-clustered in $l^{\alpha}(\mathbb{Z}^{d})$ with
$\lambda_{N}\in\mathbb{R}^{d}$ and $0<\alpha<\min\{2s-d, 2\}$.
 Then  there exists
 $C_{0}>0$ (being independent of $N$)  such that
\begin{align}\begin{array}{lll} \label{irsamapp}
\displaystyle
||f-\sum_{k\in\mathbb{Z}^{d}}f(M^{-N}(k+\varepsilon_{N,k}))\phi(M^{N}\cdot-k)||_{2}\\
 \leq
C_{0}||f||_{H^{\varsigma}(\mathbb{R}^{d})}\big[
m^{-\eta_{\kappa+1}(s, \varsigma)
N}+m^{-N\zeta}\big((1+||\lambda_{N}||_{2}^{\zeta})
||\theta_{N}||_{\hbox{m}}+||\lambda_{N}||_{2}^{\zeta} \big)\big]
\end{array}
\end{align}
holds  for every $ f\in H^{\varsigma}(\mathbb{R}^{d}),$ where $
\zeta=\min\{\varsigma-s, 1,(\frac{4s+(\alpha-2)d}{2s-\alpha+2}+d)/2
\} $, $\eta_{\kappa+1}$ is defined in Theorem \ref{Theorem x1}
\eqref{rt609}, and as in Theorem \ref{Theorem xe4},
$||\theta_{N}||_{\hbox{m}}=\max\{||\theta_{N}||_{l^{2}(\mathbb{Z}^{d})},
||\theta_{N}||^{\alpha/2}_{l^{\alpha}(\mathbb{Z}^{d})}\}$.
\end{theo}

\begin{proof}
Construct a distribution $\Delta$ on $\mathbb{R}^{d}$ by
\begin{align}\label{caiyanghanshu} \Delta(x_{1}, x_{2},
\ldots, x_{d})=\delta(x_{1}) \times\delta(x_{2})\times\cdots
\times\delta(x_{d}),\end{align} where $\delta$ is the delta
distribution on $\mathbb{R}$, and $\times$ is the tensor product. It
follows from $\widehat{\delta}\equiv1$ that $\Delta\in
H^{-t}(\mathbb{R}^{d})$ is $M$-refinable for any $t>d/2$. We suppose
here that $t$ is
 smaller than $s$.
Since $\phi$  has $\kappa+1$ sum rules, by MEP \cite[Algorithm
4.1]{Li0}, we can construct a pair of dual $M$-framelet systems
$X^{s}(\phi;\psi^{1}, \psi^{2},\ldots, \psi^{m^{d}})$ and
$X^{-s}(\Delta;\widetilde{\psi}^{1},
 \widetilde{\psi}^{2}, \ldots, \widetilde{\psi}^{m^{d}})$ such that $\widetilde{\psi}^{1},
 \widetilde{\psi}^{2}, \ldots,$ and $\widetilde{\psi}^{m^{d}}$ have $\kappa+1$ vanishing
moments. Recalling  the sampling property of $\delta$, we have
\begin{align}\label{integral} \langle f, \Delta\rangle=f(0).\end{align}
Combining  \eqref{rt1} and \eqref{integral}, the operators
$\mathcal{S}_{\phi}^{N}$ and $\mathcal{S}_{\phi;\varepsilon}^{N}$
defined in \eqref{rt} and \eqref{suanzidigyi} can be expressed by
\begin{align}\label{integral1}\mathcal{S}_{\phi}^{N}f=\sum_{k\in\mathbb{Z}^{d}}f(M^{-N}k)\phi(M^{N}\cdot-k), \ \mathcal{S}_{\phi;\varepsilon}^{N}f=\sum_{k\in\mathbb{Z}^{d}}f(M^{-N}(k+\varepsilon_{k}))\phi(M^{N}\cdot-k).\end{align}
 By Theorem \ref{Theorem x1} \eqref{bound1} and Theorem \ref{Theorem xe4} \eqref{jyy}, we obtain
\begin{align}\begin{array}{lll} \nonumber \displaystyle
||f-\sum_{k\in\mathbb{Z}^{d}}f(M^{-N}(k+\varepsilon_{N,k}))\phi(M^{N}\cdot-k)||_{2}\\
\leq ||f||_{H^{\varsigma}(\mathbb{R}^{d})}\big[ C(s, \varsigma)
m^{-\eta_{\kappa+1}(s, \varsigma)
N}+C_{3}m^{-N\zeta}\big((1+||\lambda_{N}||_{2}^{\zeta})
||\theta_{N}||_{\hbox{m}}+||\lambda_{N}||_{2}^{\zeta}
\big)\big]\\
\leq C_{0}||f||_{H^{\varsigma}(\mathbb{R}^{d})}\big[
m^{-\eta_{\kappa+1}(s, \varsigma)
N}+m^{-N\zeta}\big((1+||\lambda_{N}||_{2}^{\zeta})
||\theta_{N}||_{\hbox{m}}+||\lambda_{N}||_{2}^{\zeta} \big)\big],
\end{array}
\end{align}
where $C_{0}=\max\{C(s, \varsigma), \ C_{3}\}.$
\end{proof}

\begin{rem}
In \eqref{irsamapp},  the constant sequence $\{\lambda_{N}\}$ does
not contribute to the sampling nonuniformity. However, as mentioned
in Remark \ref{kjcw} (I), separating $\{\lambda_{N}\}$ from the
perturbation  sequence $\varepsilon_{N}$ is crucial for establishing
the sampling approximation error in \eqref{irsamapp}.
\end{rem}

\begin{rem}\label{nonu}
(I) The estimate in   \eqref{irsamapp} states that the approximation
$\sum_{k\in\mathbb{Z}^{d}}f(M^{-N}(k+\varepsilon_{N,k}))\phi(M^{N}\cdot-k)$
of $f$ is robust to the perturbation sequence
$\varepsilon_{N}=\{\varepsilon_{N,k}:=\theta_{N,k}+\lambda_{N}\}_{k\in\mathbb{Z}^{d}}$.
At every scale level $N$, if  the perturbation sequence satisfies
\begin{align}\label{tj} ||\lambda_{N}||_{2}^{\zeta}+
\max\{||\theta_{N}||_{l^{2}(\mathbb{Z}^{d})},
||\theta_{N}||^{\alpha/2}_{l^{\alpha}(\mathbb{Z}^{d})}\}
=\hbox{o}(m^{N\gamma}),\end{align}  where $\gamma<\zeta/2$, then
$\lim_{N\rightarrow\infty}\sum_{k\in\mathbb{Z}^{d}}f(M^{-N}(k+\varepsilon_{N,k}))\phi(M^{N}\cdot-k)=f$.

(II) The nonuniform sampling approximation in Theorem \ref{Theorem
x4} \eqref{irsamapp} depends on Theorem \ref{Theorem xe4} and
Theorem \ref{Theorem x2}. It follows from  Remark \ref{lubax} (II)
that the approximation in \eqref{irsamapp} is robust to the
perturbation sequence provided that $s>d/2.$

(III) In the theory of nonuniform sampling in shift-invariant spaces
(c.f. \cite{Aldroubi1,Qiyu0}), the corresponding  sample set $
X=\{x_{k}\}$ is relatively-separated. Specifically, there exists a
positive constant $\mathscr{D}(X)$ such that
\begin{align}
\label{kefen}\sum_{x_{k}\in X}\chi_{[0, 1]^{d}+x_{k}}(x)\leq
\mathscr{D}(X)
\end{align}
for any $x\in \mathbb{R}^{d}$.
 The relatively-separatedness is a natural requirement for
the finite rate of innovation of sampling \cite{Qiyu0}. For any
fixed scale level $N$,  our sampling set
$\{M^{-N}(k+\varepsilon_{N,k})\}_{k\in\mathbb{Z}^{d}}$ in Theorem
\ref{Theorem x4} is relatively-separated. Particularly, using  the
equivalence of the norms $||\cdot||_{\infty}$ and $||\cdot||_{2}$ of
$\mathbb{R}^{d}$,  it is easy to prove  that
 \eqref{kefen} holds with $X$ being replaced by
$\{M^{-N}(k+\varepsilon_{N,k})\}_{k\in\mathbb{Z}^{d}}$, and the
upper bound $\mathscr{D}(X)$ replaced  by
\begin{align}\label{shangjie} \mathscr{D}_{N}(X)=\big(\lfloor
\sqrt{d}||\lambda_{N}||_{2}+\sqrt{d}||\theta_{N}||_{l^{2}(\mathbb{Z}^{d})}+2\sqrt{d}m^{N}\rfloor\big)^{d}.\end{align}
From \eqref{shangjie}, however,  we do  not expect
$\{\mathscr{D}_{N}(X)\}_{N}$ is uniformly bounded. The underlying
reason is that Theorem \ref{Theorem x4} is on the  sampling
reconstruction of all the functions in Sobolev space
$H^{s}(\mathbb{R}^{d})$, but not just on that in   a shift-invariant
subspace.
\end{rem}

Next we make a comparison between Theorem \ref{Theorem x4} and the
existing results on sampling approximation in
$H^{s}(\mathbb{R}^{d})$.

\begin{comparison}\label{duibi}
There are some papers addressing the sampling approximation to the
functions in $H^{s}(\mathbb{R}^{d})$, see
\cite{Brown,fagep,Butzer4,Skopina0,Skopina2} and the references
therein. The approximations in the references above are carried out
by uniform samples. As mentioned in Section \ref{section1}, however,
due to the inertia of a measuring instrument, it is very difficult
to sample at an exact time. Instead, the samples we acquire may well
be jittered. Therefore, it is necessary to construct the nonuniform
sampling approximation to the functions in $ H^{s}(\mathbb{R}^{d})$.
To the best of our knowledge, the problem has not been solved in the
literature. In Theorem \ref{Theorem x4} \eqref{irsamapp}, we
constructed \textbf{a type of nonuniform sampling approximation
holding for the entire space}  $H^{s}(\mathbb{R}^{d})$. By Remark
\ref{nonu}, if the samples satisfy \eqref{tj}, then the
approximation is stable. We next concretely compare our results with
the existing ones.

The function $\phi$ in \eqref{irsamapp} can be bandlimited or
non-bandlimited. When selecting the $2$-refinable function
\cite{HanZ} $\phi(x_{1}, \ldots,
x_{d})=\prod^{d}_{j=1}\hbox{sinc}(x_{j}):=\prod^{d}_{j=1}\frac{\sin\pi
x_{j}}{\pi x_{j}}$, then it follows from \eqref{irsamapp} that
\begin{align}\label{sincapp} f(x_{1},
x_{2}, \ldots,
x_{d})\approx\sum_{k\in\mathbb{Z}^{d}}f(2^{-N}(k+\varepsilon_{N,k}))\prod^{d}_{j=1}\hbox{sinc}(2^{N}x_{j}-k_{j}).\end{align}
Moreover, if $d=1$ and
\begin{align} \label{unifr}\varepsilon_{N}=\{\varepsilon_{N,k}\}=0, \end{align}
then the sampling approximation results for $H^{s}(\mathbb{R})$ in
\cite{Brown,Butzer4,Skopina0} is revisited. When \eqref{unifr} holds
and $f$ satisfies
$$|\widehat{f}(\xi)|\leq C(1+||\xi||_{2})^{\frac{-d-\alpha}{2}} \ \hbox{for} \ \hbox{every}\ \xi\in \mathbb{R}^{d}$$
with $\alpha>0$ and  a constant  $C$ being dependent on $f$, then
the approximation in \eqref{irsamapp} reduces to the results in
\cite{fagep}.

Using a function $\phi$ satisfying some orders of Strang-Fix
condition, Krivoshein and  Skopina \cite{Skopina2} constructed the
approximation to smooth functions by the uniform samples of
functions and their derivatives. The nonuniform sampling in
\eqref{irsamapp} holds for all the functions in $
H^{s}(\mathbb{R}^{d})$ where $s>d/2.$ For any $f\in
H^{s}(\mathbb{R}^{d})$, it is not necessary smooth. For example, the
box spline $B_{\Xi}(x_{1}, x_{2})=B_{2}(x_{1})B_{2}(x_{2})$ is not
smooth, and by \cite{Hanbin1}, $B_{\Xi}\in H^{s}(\mathbb{R}^{2})$
with $1<s<3/2,$ where $B_{2}$ is the cardinal B-spline of order $2$,
defined by
\begin{align}\label{box}
B_{2}(t)= \left\{
 \begin{array}{ll}
 t, & \hbox{ $t\in [0,1)$} \\
 2-t, & \hbox{ $t\in [1,2]$} \\
 0, & \hbox{ else}
  \end{array}.
  \right.
\end{align}

%

\end{comparison}

\section{Numerical experiment---randomly jittered  sampling}
In this section, numerical experiments are carried out to confirm
the efficiency of our sampling approximation formula Theorem
\ref{Theorem x4} \eqref{irsamapp}.   Provided that \eqref{tj} holds,
the  sequence
$\varepsilon_{N}=\{\varepsilon_{N,k}:=\theta_{N,k}+\lambda_{N}\}_{k\in\mathbb{Z}^{d}}$
in \eqref{irsamapp} is supposed to be random for avoiding the bias
toward perturbation.

\subsection{One dimension}  \label{1wei}
Let $f(x)=e^{-|x|}, x\in\mathbb{R}.$ It is not smooth, and its
Fourier transform $\widehat{f}(\xi)=\frac{2}{1+(2\pi \xi)^{2}}.$
Obviously, $f\in H^{s}(\mathbb{R})$, where $1/2<s<3/2.$   In this
subsection, we use \eqref{irsamapp} with $\phi=\hbox{sinc}$ to
approximate $f$ on $[-40, 40],$ where $\lambda_{N}$ and
$\theta_{N,k}, k\in \mathbb{Z}$ are independent, and obey the
uniform distribution on $[-1, 1]$. Specifically,
\begin{align} \label{bjgs} f\approx \sum^{89\times 2^{N}}_{k=-89\times 2^{N}}f(2^{-N}(k+\varepsilon_{N,k}))\hbox{sinc}(2^{N}\cdot-k),\end{align}
and the relative error is defined as
\begin{align}  \hbox{error}=||f-\sum^{89\times 2^{N}}_{k=-89\times 2^{N}}f(2^{-N}(k+\varepsilon_{N,k}))\hbox{sinc}(2^{N}\cdot-k)||_{2}/||f||_{2}.\end{align}
For $N=10$, the  formula  \eqref{bjgs} is carried out to approximate
$f$ for $30$ times. See Figure \ref{figure4_1} for the error
distribution.

\begin{figure}\label{figure4_1}
    \centering
\includegraphics[width=13cm, height=8cm]{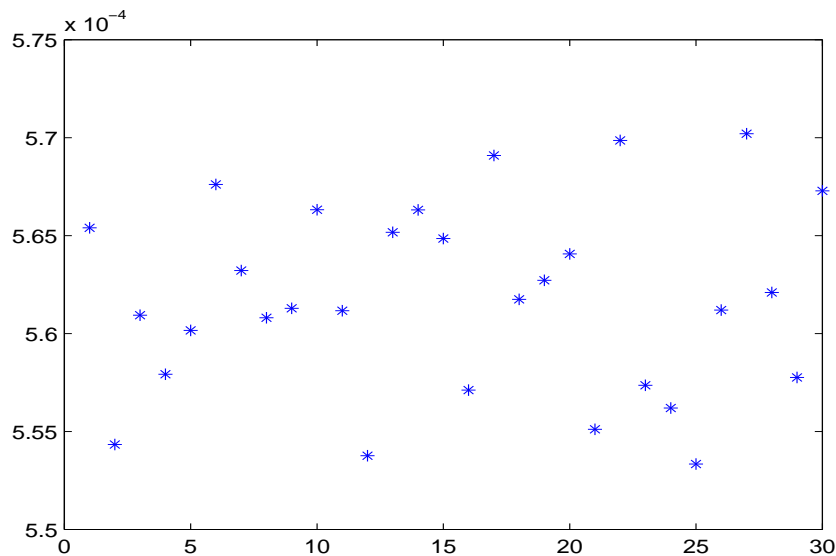}
    \caption{The error distribution when $N=10$.}
\end{figure}

\subsection{Two dimensions}
Let $\phi(x_{1}, x_{2})=B_{\Xi}(x_{1},
x_{2})=B_{2}(x_{1})B_{2}(x_{2})$, where $B_{2}$ is the cardinal
B-spline of order $2$ defined in \eqref{box}. By \cite{Hanbin1},
$\phi(x_{1}, x_{2})$ is $2$-refinable and $\phi\in
H^{s}(\mathbb{R}^{2})$ with $1<s<3/2.$ Suppose that
$$f(x_{1}, x_{2})=e^{-(|x_{1}|+|x_{2}|)}+e^{-(x_{1}^{2}+x_{2}^{2})}.$$
It follows from Subsection \ref{1wei} that $e^{-(|x_{1}|+|x_{2}|)}
\in H^{s}(\mathbb{R}^{2})$ where $ s\in (1, 3/2)$. On the other
hand, $e^{-(x_{1}^{2}+x_{2}^{2})}\in H^{s}(\mathbb{R}^{2})$ for any
$s\in \mathbb{R}^{+}$. Therefore $f\in H^{s}(\mathbb{R}^{2}), s\in
(1, 3/2)$. We next use \eqref{irsamapp} to approximate $f$ on $[-20,
20]^{2}$. That is,
\begin{align} \label{56} f|_{[-20, 20]^{2}} \approx\sum_{k\in \mathbb{Z}^{2}}f(2^{-N}(k+\varepsilon_{N,k}))\phi(2^{N}\cdot-k)\mid_{[-20, 20]^{2}}.\end{align}
The corresponding relative error is defined as
\begin{align} \label{bjgs1} \hbox{error}=||f|_{[-20, 20]^{2}}-\sum_{k\in \mathbb{Z}^{2}}f(2^{-N}(k+\varepsilon_{N,k}))\phi(2^{N}\cdot-k)\mid_{[-20, 20]^{2}}||_{2}/||f|_{[-20, 20]^{2}}||_{2}.\end{align}
Since $\phi$ is compactly supported, the series in \eqref{56} is
actually involved with finite sums. On the other hand, $\lambda_{N}$
and $\theta_{N,k}, k\in \mathbb{Z}^{2}$ are independent, and obey
the uniform distribution on $[-1, 1]^{2}$. When $N=10$, the
approximation formula in \eqref{56} is carried out for $30$ times.
See Figure \ref{figure42} for the error distribution.

\begin{figure}\label{figure42}
    \centering
\includegraphics[width=13cm, height=8cm]{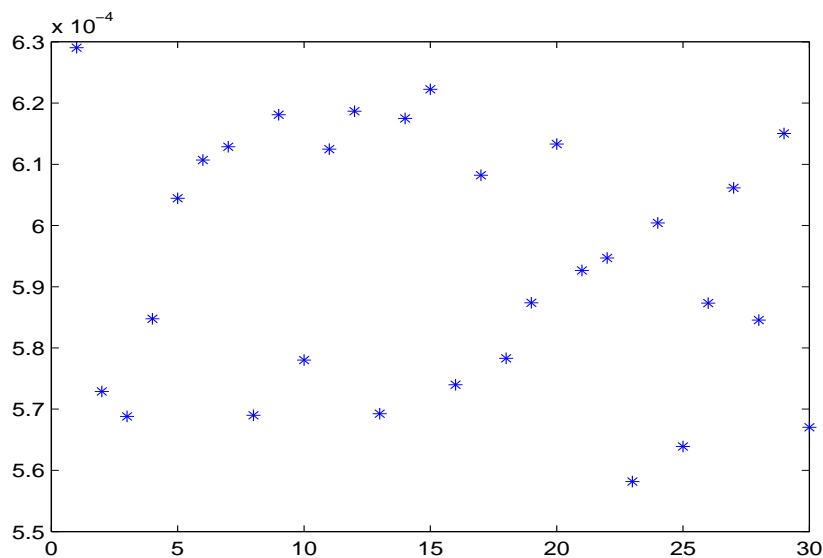}
    \caption{The error distribution when $N=10$.}
\end{figure}

\end{document}